\documentclass[12pt,a4paper,reqno]{amsart}
    \usepackage{enumitem}
    \usepackage{amsmath}    
    \usepackage{amssymb}
    \usepackage{tikz}
    \usepackage{tikz-cd}
    \usepackage{graphicx} 
    \usepackage{subcaption}
    \usepackage[T1, T2A]{fontenc}
    \usepackage[greek, russian, english]{babel}
    \usepackage[colorlinks=true, linkcolor=blue]{hyperref}
    \usepackage{orcidlink}
    \usepackage{enumitem}
    \title[Discrete Hardy spaces on the half-lattice]{Hardy spaces of discrete holomorphic functions \\on the upper half-lattice}
    \usepackage{xcolor}
    \date{}
    \usepackage{amsfonts, amsthm, mathtools}
    \usepackage[a4paper,left=2cm,right=2cm]{geometry}

    \DeclareMathOperator*{\supp}{supp}

    \newtheorem{thm}{Theorem}[section]
    \newtheorem{lem}[thm]{Lemma}
    \newtheorem{prop}[thm]{Proposition}
    \newtheorem{cor}[thm]{Corollary}
    
    \theoremstyle{definition}
    \newtheorem{rmk}[thm]{Remark}

    \theoremstyle{definition}
    \newtheorem*{thm*}{Theorem}

    \usepackage{tcolorbox}
    \newcommand{\nb}[3]{{\colorbox{#2}{\bfseries\sffamily\scriptsize\textcolor{white}{#1}}} 
    {\textcolor{#2}{\sf\small\textit{#3}}}}
            	\definecolor{jazzberryjam}{rgb}{0.65, 0.04, 0.37}
    \newcommand{\eug}[1]{\nb{Eug}{jazzberryjam}{#1}}
    \newcommand{\ale}[1]{\nb{Ale}{teal}{#1}}

    \allowdisplaybreaks
    \author[E. Dellepiane]{Eugenio Dellepiane \orcidlink{0009-0005-3432-2053}}
    \address{D\'epartement de math\'ematiques et de statistique,
    Universit\'e Laval, Qu\'ebec, QC, Canada G1V 0A4}
    \email{dellepianeeugenio@gmail.com}
    \author[A. Monguzzi]{Alessandro Monguzzi \orcidlink{0000-0003-3233-5000}} 
    \address{Dipartimento di Ingegneria Gestionale, dell'Informazione e della
    Produzione, Universit\`a degli Studi di Bergamo, viale Marconi 5,
    24044 Dalmine, Italy}
    \email{alessandro.monguzzi@unibg.it}
    \author[M. Monti]{Matteo  Monti\orcidlink{0000-0001-6848-5938}}
    \address{Dipartimento di Matematica e Applicazioni, Universit\`a degli Studi di Milano--Bicocca,
    Via Cozzi 55, 20125 Milano, Italy}
    \email{matteo.monti@unimib.it}
    \subjclass[2020]{Primary 30G25, 30H10; Secondary 39A12, 31C20}
    \thanks{All the authors are members of Indam--Gnampa. E. Dellepiane acknowledges the support of the CRM-Laval post-doctoral fellowship and the Alliance Grant. M. Monti is partially supported by the GNAMPA–INdAM Project 2026 "Transferring Harmonic Analysis between Discrete Structures and Manifolds" (CUP E53C25002010001).} 
     \keywords{Discrete holomorphic functions, Hardy spaces, Paley--Wiener theorem, Bergman spaces}
    
\begin{document}
    \begin{abstract}
    We develop a theory of Hardy spaces $H^p$ of discrete holomorphic functions on
    the upper half-lattice, within the classical framework of discrete
    holomorphicity on the square lattice. We prove Cauchy and Poisson reproducing
    formulas, establish a boundary norm identity, and obtain Paley--Wiener type
    characterizations for these spaces. In the Hilbert space case, we describe the
    associated reproducing kernel and Szeg\H{o} projection, and we compare the
    discrete theory with the classical Hardy space on the upper half-plane through
    a family of discrete holomorphic approximants of classical $H^2$-functions. We
    also prove duality results for $H^p$, $1<p<\infty$, establish uniqueness and
    sampling results on horizontal lines, and introduce Bergman-type spaces,
    comparing two natural weighted scales.
    \end{abstract}
    \maketitle

    

    \section{Introduction and main results}
    
    Discrete complex analysis has a long history. Early forms of discrete
    holomorphicity appeared in the work of Isaacs \cite{Isa41,Isa52} and Ferrand
    \cite{F}, and were later developed systematically by Duffin and
    collaborators in the setting of the square lattice and of rhombic lattices
    \cite{D,Duf68,DP68}. In this framework, discrete holomorphic functions
    satisfy finite-difference analogues of the Cauchy--Riemann equations and retain
    several features of classical holomorphic functions, such as Cauchy-type
    formulas and close connections with harmonic functions and potential theory.
    After these early developments, the subject remained relatively quiet for some
    time, but it has attracted renewed interest in the last decades. Discrete
    complex analytic methods now play an important role in probability, statistical
    mechanics, and the analysis of critical lattice models, especially through the
    work of Kenyon \cite{Ken02}, Smirnov, Chelkak, and collaborators
    \cite{CS11,CS12,Che16,CLR23}. Related developments include discrete Riemann
    surfaces and discrete complex structures on graphs and surfaces
    \cite{Mer01,Mer08,BG16,Sko13}. For general introductions to discrete complex
    analysis we refer the reader to the expository works \cite{Lov04,Smi10} and to
    the references therein.
    
    In contrast with these developments, the function space theory of discrete
    holomorphic functions is still much less developed. One basic obstruction is
    that the pointwise product of two discrete holomorphic functions is not, in
    general, discrete holomorphic. Thus, several tools from the classical theory of holomorphic function spaces, such as factorization or Blaschke products, do not have
    immediate discrete counterparts. This suggests that a theory of holomorphic function spaces in the discrete setting cannot be obtained by a direct
    translation of the continuous one, but has to be built from the specific
    analytic and geometric structure of the lattice. 
    
    Some foundational results in this direction are already present in Duffin's
    work. In \cite{D}, Duffin introduced discrete contour integrals, proved
    Cauchy-type formulas, and studied the fundamental kernel associated with the
    square lattice. Later, Zeilberger, Dym, and collaborators developed Hilbert
    space methods for discrete analytic functions and obtained some Paley--Wiener
    type results \cite{ZD,Zei77a,Zei77b,Zei77c}. Other related contributions are
    scattered in the literature, for instance \cite{Bol68,DM71,Mug80}. More
    recently, a systematic study of Paley--Wiener spaces of discrete holomorphic
    functions on the square lattice was initiated in \cite{MM}. However, apart from
    these contributions, the function space theory of discrete holomorphic
    functions remains largely undeveloped. In particular, a systematic Hardy space
    theory for discrete holomorphic functions on the half-lattice does not seem to
    have been developed.
    The main purpose of the present paper is to provide such a framework and to offer a starting point for a broader function space theory of discrete holomorphic functions.
    
    Throughout the paper we identify the lattice $\mathbb Z^2$ with $\mathbb Z+i\mathbb Z\subseteq\mathbb C$, and freely use the notation $m+in$ for the lattice point $(m,n)$. We focus on the discrete upper half-lattice
    \[
    \mathbb Z^2_+
    =
    \{m+in\in\mathbb Z^2:\ n\geqslant0\},
    \]
    and study Hardy spaces of discrete holomorphic functions on $\mathbb Z^2_+$.

    For $1\leqslant p\leqslant\infty$ we define
    \[
    H^p(\mathbb{Z}^2_+)
    =
    \left\{
    f\in \operatorname{Hol}(\mathbb Z^2_+):
    \sup_{n\geqslant0}\|f(\cdot+in)\|_{\ell^p(\mathbb Z)}<\infty
    \right\}.
    \]
    
    The Hilbert space case $p=2$ was already considered by Zeilberger and 
    Dym~\cite{ZD}. Their approach is essentially spectral and is tied to the Hilbert space setting: the space $H^2$ is 
    identified, through the Fourier transform, with a subspace of $L^2(\mathbb T)$, 
    and its 
    main properties are then deduced from this model. Our point of view is different. Starting from the intrinsic discrete analytic
    structure, and in particular from Duffin's Cauchy-type formula, we derive the
    basic reproducing formulas directly in the Banach setting
    $1\leqslant p<\infty$. The Paley--Wiener characterization is then recovered 
    as a consequence of the Cauchy reproducing formula.
    
    A key step in this construction is to show that the $H^p$-norm is determined by the boundary values. More precisely, we prove the norm identity
    \begin{equation}\label{eq:norm_boundary_identity}
    \|f\|_{H^p}
    =
    \|f|_{\mathbb Z}\|_{\ell^p(\mathbb Z)},
    \qquad 1\leqslant p<\infty,
    \end{equation}
    see Proposition~\ref{p:H^p_norm_2}. This identity allows us to transfer questions about $H^p$-functions to questions
    about their boundary restrictions in $\ell^p(\mathbb Z)$.   
    It is worth mentioning that an alternative notion of Hardy space for discrete analytic functions was introduced by Alpay and Volok in \cite{AV}. Their construction focuses on the Hilbert space setting and is based on expansions with respect to discrete analytic polynomials, modeling the classical Hardy space on the unit disk. By contrast, the Hardy spaces considered here are modeled on the classical Hardy spaces on the upper half-plane and are characterized by boundary values and discrete Fourier analysis.

    The proof of \eqref{eq:norm_boundary_identity} is not trivial. Starting from Duffin's discrete Cauchy-type formula, we derive a Cauchy reproducing
    formula in the half-lattice for $H^p$ functions; see Proposition~\ref{p:cauchy_half_plane}. We then
    identify the real part of this kernel with a Poisson-type kernel for the discrete
    upper half-lattice and obtain a Poisson-type reproducing formula; see
    Proposition~\ref{p:reproducing_poisson}. Together with the positivity and normalization of the Poisson kernel, this yields the boundary norm identity. 
    For the corresponding boundary subspaces we obtain Paley--Wiener-type characterizations; see
    Theorem~\ref{t:H^2_PW} for the Hilbert space case and
    Theorem~\ref{t:H^p_PW} for the general $H^p$ case. Again, in the Hilbert space case,
    this result was already proved in \cite{ZD}, but from a different perspective; see the remarks after Theorem~\ref{t:H^2_PW}. We also describe the
    corresponding Szeg\H{o} projection. On the boundary, it is a convolution
    operator with Duffin's Cauchy kernel and can be written in terms of a discrete
    Hilbert-transform type operator; see \eqref{Eq:P_Hilbert}. Thus the classical
    relation between Hardy space boundary values, the Cauchy kernel, the
    Szeg\H{o} projection, and the Hilbert transform has a precise analogue in the
    present discrete setting. We also relate the discrete theory to the classical Hardy space on the upper half-plane. More precisely, we show that every classical $H^2$-function admits a natural family of discrete holomorphic $H^2$-approximants on rescaled half-lattices, and that these approximants converge uniformly on compact subsets away from the boundary; see Section \ref{s:approximation}. 
    
    In the last part of the paper we record some further consequences of the
    boundary structure and introduce Bergman-type spaces. We prove uniqueness and sampling
    results  showing in particular that the parity structure of the square lattice
    plays a visible role. The even and odd boundary sublattices are sampling sets
    for $H^p$, $1<p<\infty$, and in the Hilbert case the $H^2$-norm splits exactly
    between them. This boundary sampling phenomenon has no direct analogue in the
    classical Hardy space on the upper half-plane; see Propositions~\ref{p:sampling},
    \ref{p:sampling_Hp}, and \ref{p:not_sampling_positive_height}.

     We introduce Bergman-type spaces of
    discrete holomorphic functions. The unweighted Bergman space is naturally defined, and its Paley--Wiener representation and reproducing kernel follow directly from the Hardy space theory; see
    Theorem~\ref{t:B^2_PW} and Corollary~\ref{c:Bergman_kernel}. This suggests
    natural families of weighted Bergman spaces, and we compare two such weighted
    scales: one defined by powers of the diagonal Bergman kernel and one defined
    through the spectral weight appearing in the Paley--Wiener representation.
    These two scales identify the same spaces of functions, up to
    equivalence of norms.
    
    The paper is organized as follows. In Section \ref{s:preliminaries} we recall the necessary preliminaries on harmonic analysis and discrete holomorphicity. In Section \ref{s:Hp} we prove the Cauchy and Poisson reproducing formulas and, as a consequence, the norm identity \eqref{eq:norm_boundary_identity}. Section \ref{s:H2} is devoted to the Hilbert space structure of $H^2$, the Paley--Wiener theorem, the Szeg\H o projection, and related consequences. In Section \ref{s:approximation} we compare the discrete theory with the classical Hardy space on the upper half-plane by constructing discrete approximants of classical $H^2$-functions. In Section \ref{s:Hp_duality} we prove the Paley–Wiener characterization and duality results for the spaces $H^p$, $1<p<\infty$. Section \ref{s:sampling} contains preliminary results on uniqueness and sampling sets. Section \ref{s:Bergman} introduces Bergman-type spaces and weighted Bergman scales. We conclude in Section \ref{s:final} with some final remarks and possible directions for future research.

    \section{Preliminaries}\label{s:preliminaries}
    We recall in this section the preliminary material in harmonic analysis and discrete holomorphicity that will be used throughout the paper. The harmonic analysis part fixes our conventions for the Fourier transform and periodic distributions. The discrete holomorphic material is mostly classical, see \cite{D}, but less standard. Therefore we present it in some detail in order to keep the paper reasonably self-contained.
    \subsection{Fourier transform and distributions on the torus}\label{ss:distributions}
    Throughout the paper we  identify the standard torus $
    \mathbb T=\mathbb R/2\pi\mathbb Z$ 
    with the interval $[-\pi,\pi]$. We then consider on $[-\pi,\pi]$ the Lebesgue spaces $L^p$ with respect to the  unnormalized Lebesgue measure. In particular, $
    \left\{(2\pi)^{-1/2}e^{im(\cdot)}\right\}_{m\in\mathbb Z}
    $ 
    is an orthonormal basis of $L^2(\mathbb T)$.
    
    The Fourier transform is a fundamental tool in our analysis. We denote by $\mathcal F$ the mapping
    \begin{align*}
    \mathcal{F}\colon \ell^2(\mathbb{Z})&\to L^2([-\pi,\pi])\\
    f &\mapsto (\mathcal F f)(t)=\frac{1}{\sqrt{2\pi}}\sum_{m=-\infty}^{+\infty}f(m) e^{-imt},\quad t\in[-\pi,\pi],
    \end{align*}
    with inverse
    \[
    (\mathcal{F}^{-1}g)(m)=\frac{1}{\sqrt{2\pi}}\int_{-\pi}^{\pi}g(t)e^{imt}\operatorname{d}\!t, \qquad m\in\mathbb{Z}.
    \]
    We use this choice of signs and normalization throughout the paper; with it, $\mathcal F^{-1}=\mathcal F^*$. We will need to use the \emph{distributional} Fourier transform for $\ell^p$-functions. Thus, we gather some basic facts about the distributions over the torus. For references, we include \cite{beals,distributionsRT,hormander1983analysis}. By $\mathcal{P}:=C^\infty(\mathbb{T})$ we denote the set of smooth functions on $\mathbb{R}$ that are $2\pi$-periodic. We consider the standard topology on $C^\infty$, that is, $\psi_n\to\psi$ in $C^\infty(\mathbb{T})$ if and only if $\psi_n^{(k)}$, the derivative of order $k$ of $\psi_n$, converges uniformly to $\psi^{(k)}$ as $n\to\infty$, for every order $k\geqslant0.$ This space can be characterized in terms of the Fourier series~\(\mathcal{F}^{-1}.\)
    We denote by $\mathcal{S}(\mathbb{Z})$ the space of \emph{rapidly decaying}  functions, that is,
    \[
    \mathcal S(\mathbb Z)=\Big\{\varphi:\mathbb Z\to \mathbb C\textrm{ s.t. }\forall M\in\mathbb N,\,  p_M(\varphi)=\sup_{k\in\mathbb{Z}}(|k|+1)^{M}|\varphi(k)|<\infty \Big\}.
    \]
    Then, $\mathcal{S}(\mathbb{Z})$ is a Fréchet space, with respect to the topology induced by the family of norms $\{p_M\}_{M\in\mathbb{N}}$. 
     The set of the (tempered) distributions $\mathcal{S}'(\mathbb{Z})$ is defined as the dual of $\mathcal{S}(\mathbb{Z})$. By testing on the functions $\operatorname{e}_n$ given by $\operatorname{e}_n(k)=\delta_{n,k}$ and setting $u(n):=u(\operatorname{e}_n)$, one can check that every $u\in\mathcal{S}'(\mathbb{Z})$ is a pointwise well-defined function on $\mathbb{Z}$ with \emph{slow growth}, that is,
     \[
     |u(n)|\leqslant C(|n|+1)^r,
     \]
     for some positive $C,r$, and the action of $u$ on $\mathcal{S}(\mathbb{Z})$ is given by
    \[
    \varphi \mapsto u(\varphi)=\sum_{n=-\infty}^{+\infty} u(n)\varphi(n), \qquad \varphi\in\mathcal{S}(\mathbb{Z}).
    \]
    A classic result of harmonic analysis bridges these two different worlds: the higher the regularity of $\psi$, the faster the decay of its Fourier coefficients. In symbols, $\psi\in C^\infty(\mathbb{T})$ if and only if $\mathcal{F}^{-1}\psi\in\mathcal{S}(\mathbb{Z})$. Dualizing $\mathcal{F}^{-1}$, one obtains a bijection $\mathcal{F}\colon \mathcal{S}'(\mathbb{Z})\to\mathcal{P}'$, that acts as

    \begin{equation}\label{E:distributionalFourier}
         \mathcal{F}u(\psi) =  u(\mathcal{F}^{-1}\psi) =\sum_{n=-\infty}^{+\infty} u(n) \mathcal{F}^{-1}\psi(n), \qquad \psi\in\mathcal{C}^\infty(\mathbb{T}),u\in\mathcal{S}'(\mathbb{Z}).
    \end{equation}

    This is how we shall understand the Fourier transform of $\ell^p$-functions, $1\leqslant p\leqslant \infty$, being functions on $\mathbb Z$ with slow growth. If $p\leq2$, this definition coincides with the classical Fourier transform. 
    
    We shall use the standard notion of support of a distribution on the torus. In particular,
    $\operatorname{supp}\mathcal F f\subseteq[0,\pi]$ means that
    $\langle\mathcal F f,\psi\rangle=0$ for every test function $\psi$ supported in
    $(-\pi,0)$.

    \subsection{Discrete holomorphicity}
    We review Duffin’s discrete contour integrals and the corresponding Cauchy-type formula; see \cite{D}. We begin with some notation and basic definitions. Recall that we always identify $\mathbb Z^2$ with $\mathbb Z+i\mathbb Z\subseteq \mathbb C$.
    
    Given $z_0$ in $\mathbb Z^2$ we set $z_1=z_0+1, z_2=z_0+1+i$ and $z_3=z_0+i$. 
    Consider the unit square
    
    \bigskip
    
    \begin{center}
    \begin{tikzpicture}[scale=0.5]
        \def\gridsize{3}
        
        \foreach \x/\y/\label in {0/0/$z_0$, 3/0/$z_1$, 3/3/$z_2$, 0/3/$z_3$}
        {
            \fill (\x, \y) circle (2pt) node[right=2pt] {\label};
        }
    \end{tikzpicture}
    \end{center}
    \bigskip
    and set $f_j=f(z_j)$. We refer to this square as the \emph{square associated with $z_0$}.  For $\varepsilon$ in $\{\pm 1,\pm i\}$ we introduce the operators  
    \begin{align*}
    X_{\varepsilon} f(z_0)= f_0+\varepsilon f_1+\varepsilon^2f_2+\varepsilon^3f_3.
    \end{align*}
    Sometimes, with a slight abuse of notation, if $Q$ is a square, by evaluating $X_\varepsilon f$ on the square $Q$ we mean $ X_\varepsilon f(z_0)$, where $z_0$ is the south-west vertex of $Q$.
    By definition, we say that a function $f$ is (discrete) holomorphic in the square associated with $z_0$ if and only if $X_if(z_0)=0$, which is equivalent to saying that 
    \[
    \frac{f_2-f_0}{z_2-z_0}=\frac{f_3-f_1}{z_3-z_1}.
    \]
    It is not hard to check that  
    \[
    4f(z_0)=(X_i+X_{-i}+X_1+X_{-1})f(z_0).
    \]
    By rewriting the operators $X_\varepsilon$ in terms of suitable discrete line integrals, we obtain a Cauchy-type formula.    
    
    Before turning to contour integrals, we introduce the following basic family of discrete entire functions, that is, functions which are discrete holomorphic on the whole lattice. These functions and related ones were studied in \cite{F} and \cite{ZD}. For $t\in[-\pi,\pi], t\neq \pm\pi/2$, we define the \emph{discrete exponential}
    \begin{equation}\label{d:discrete-exponential}
    e_{t}(m+in) = e^{itm}\bigg(\frac{1+ie^{it}}{i+e^{it}}\bigg)^n=e^{itm}\bigg(\frac{\cos t}{1+\sin t}\bigg)^n,\qquad m,n\in\mathbb{Z}.
    \end{equation}

    For every $t$ in $[-\pi,\pi], t\neq \pm \pi/2$, the function $e_t$ is holomorphic on $\mathbb{Z}^2$ and it is an extension to $\mathbb Z^2$ of the classical exponential function. Indeed, $e_t(m)=e^{itm}$ for every $m$ in $\mathbb{Z}$.
    
    \subsection{Contour integrals}
    Let us consider the chain of lattice points $z_0,z_1,\ldots,z_m$.  Following Duffin's notation in \cite{D}, define
    \begin{align*}
     &\int_{z_0}^{z_m}f:g=\sum_{j=1}^{m}(z_j-z_{j-1})\Big(\frac{f_j+f_{j-1}}{2}\Big)\Big(\frac{g_j+g_{j-1}}{2}\Big);\\
     &\int_{z_0}^{z_m}f:g'=\sum_{j=1}^{m}\Big(\frac{f_j+f_{j-1}}{2}\Big)(g_{j}-g_{j-1});\\
     & \int_{z_0}^{z_m}f:g''= -\sum_{j=1}^{m}(z_j-z_{j-1})(f_j-f_{j-1})(g_j-g_{j-1})
    \end{align*}
    We remark that the third type of line integral can be viewed as a line integral of the first type. Given a function $f$, its \emph{dual function} $f^*$ is defined by
    \begin{equation}\label{D:dualfunction}    f^\ast(m+in)=(-1)^{m+n}\overline{f(m+in)}.
    \end{equation}
    One can check that the holomorphicity of $f^*$ on a unit square is equivalent to that of $f$, and that
    \[
    \int_{z_0}^{z_m}f:g''=-4\overline{\int_{z_0}^{z_m}f^\ast:g^\ast}.
    \]
    
    In the definitions of line integrals, we also allow closed paths, that is, $z_m=z_0$. In this case, by rearranging the terms of the previous expressions and using the convention $z_{m+1}=z_1$, one obtains 
    \begin{align*}
    &\int_{z_0}^{z_m}f:g=\sum_{j=1}^{m}\frac{f_j}{4}\big((z_j-z_{j-1})g_{j-1}+(z_{j+1}-z_{j-1})g_j+(z_{j+1}-z_j)g_{j+1}\big);\\
     &\int_{z_0}^{z_m}f:g'=\sum_{j=1}^{m}\frac{f_j}{2}(g_{j+1}-g_{j-1});\\
     &\int_{z_0}^{z_m}f:g''=\sum_{j=1}^{m}f_j\bigg(\frac{g_{j+1}-g_{j}}{z_{j+1}-z_{j}}-\frac{g_j-g_{j-1}}{z_j-z_{j-1}}\bigg).
    \end{align*} 
    

     
    By a  \emph{(simple) region} $R$ we mean a finite union of unit squares whose boundary $\partial R$ is a simple closed polygonal curve consisting of edges of the constituent squares. We say that a function $f$ is holomorphic on $R$ if it is holomorphic on each of the unit squares constituting $R$.

    The following proposition relates the values of the functions on the boundary of a region with the values of the functions in the interior of the region.   With the symbol $\sum_{Q\subseteq R}$ we mean the sum  over all unit squares  whose union is $R$. 
    \begin{prop}\label{p:integrals_identities}
       The following identities hold true.
    \begin{align}
    \label{E:integrals_identities1}&\int_{\partial R} f:g= \frac{1-i}{8}\sum_{Q\subseteq R} (X_1gX_if+X_1fX_ig);\\
    \label{E:integrals_identities2}&\int_{\partial R}f:g'=\frac{i}{4}\sum_{Q\subseteq R}(X_{-i}gX_{i}f-X_{-i}fX_{i}g);\\
    \label{E:integrals_identities3}&\int_{\partial R}f:g''=-\frac{1+i}{2}\sum_{Q\subseteq R} (X_{-1}gX_{i}f+X_{-1}fX_{i}g).
    \end{align}
    \end{prop}
    \begin{proof}
    We refer the reader to \cite[Section 3]{D}.
    \end{proof}
    
    From this proposition some Cauchy-type formulas can be derived. To this end, we introduce the following auxiliary function.  Let $q$ be a function defined on the lattice such that
    \begin{equation}\label{eq:q}
     X_iq=\delta_0,
    \end{equation}
    where $\delta_0$ is the Dirac delta.
    The existence of such a function $q$ is proved in \cite{D} and it will be briefly discussed in Subsection \ref{s:exisrtence_q}. For each lattice point $z_0$ the translated function $q_{z_0}(z):=q(z-z_0)$ satisfies the analogous property of $q$ but at $z_0$. Namely, $X_iq_{z_0}$ is the Dirac delta $\delta_{z_0}$. Thus, the function $q$ can be regarded as a discrete fundamental solution of the operator $X_i$.

    \begin{prop}\label{p:integrals_q}
    Let $R$ be a simple region and let $z_0$ be a lattice point such that its associated square is contained in $R$. Let $f$ be a function holomorphic in $R$. Then,
    \begin{align*}
       &\int_{\partial R} f:q_{z_0}=\frac{1-i}{8}X_{1}f(z_0);\\
       &\int_{\partial R}f:q_{z_0}'=-\frac{i}{4}X_{-i}f(z_0);\\
       &\int_{\partial R}f:q_{z_0}''=-\frac{1+i}{2}X_{-1}f(z_0).
       \end{align*}
    \end{prop}
    
    \begin{proof}
    The identities are an immediate consequence of Proposition \ref{p:integrals_identities}, the holomorphicity of $f$ and the definition of $q_{z_0}$.
    \end{proof}
    \begin{cor}\label{l:cauchy_formula}
    Let $R$ be a simple region and let $z_0$ be a point such that its associated square is contained in $R$. Let $f $ be a function holomorphic in $R$. We have the Cauchy-type formula
    \begin{equation}\label{eq:cauchy}
    f(z_0)=\int_{\partial R} f:\big((1+i)q_{z_0}+iq_{z_0}'-\frac{1-i}{4}q_{z_0}''\big).    
    \end{equation}
    \end{cor}
    \begin{proof}
    Since
    \[
    4f(z_0)=(X_1+X_{-1}+X_{-i}+X_i)f(z_0),
    \]
    the holomorphicity of $f$ in $R$ and Proposition \ref{p:integrals_q} guarantee that 
    \[
    4f(z_0)=\frac{8}{1-i}\int_{\partial R}f:q_{z_0}-\frac{2}{1+i}\int_{\partial R}f:q_{z_0}''-\frac{4}{i}\int_{\partial R} f:q_{z_0}'.
    \]
    The conclusion follows.
    \end{proof}
    
    Similar formulas recovering the values of $f$ at the remaining vertices
    $z_1,z_2,z_3$ of the unit square associated with $z_0$ can be obtained.
    We refer the reader to \cite[Section 3]{D}.

    \subsection{Holomorphicity and harmonicity}\label{s:harmonicity}
    
    We recall the relation between discrete holomorphicity and discrete harmonicity. For details, we refer the reader to \cite[Section 2]{D}. We denote by
    \[
    \mathbb Z^2_{\mathrm{even}}
    :=
    \{m+in\in\mathbb Z^2:\ m+n\in 2\mathbb Z\}, \qquad \mathbb Z^2_{\mathrm{odd}}
    :=
    \{m+in\in\mathbb Z^2:\ m+n\in 2\mathbb Z+1\}
    \]
    
    the even and odd sublattices of \(\mathbb Z^2\), respectively.
    
    Let \(f\) be a discrete holomorphic function in a region \(R\), and assume that
    \(R\) contains the following configuration of points,
    
    \begin{center}
    \begin{tikzpicture}[scale=0.3]
        \foreach \x/\y/\label in {
            0/0/$z_0$,
            5/0/$z_1$,
            5/5/$z_2$,
            0/5/$z_3$,
            -5/5/$z_4$,
            -5/0/$z_5$,
            -5/-5/$z_6$,
            0/-5/$z_7$,
            5/-5/$z_8$
        }
        {
            \fill (\x,\y) circle (2.5pt) node[right=3pt] {\label};
        }
    \end{tikzpicture}
    \end{center}
    Setting \(f_j:=f(z_j)\), using the holomorphicity of $f$ in the squares having for south-west vertices the points $z_0,z_5,z_6$ and $z_7$ one can see that
    \begin{equation*}
    f_0=\frac{f_2+f_4+f_6+f_8}{4}.
    \end{equation*}
    Equivalently,
    \[
    \Delta_{\mathrm{diag}} f(z_0)
    :=f_0-
    \frac{f_2+f_4+f_6+f_8}{4}
    =
    0.
    \]
    Since \(\Delta_{\mathrm{diag}}\) has real coefficients, the same identity holds
    separately for the real and imaginary parts of \(f\). The operator \(\Delta_{\mathrm{diag}}\) preserves the decomposition
    \[
    \mathbb Z^2=\mathbb Z^2_{\mathrm{even}}\cup \mathbb Z^2_{\mathrm{odd}},
    \]
    and is the usual nearest-neighbor discrete Laplacian on each of the two
    diagonal sublattices. Consequently, if \(f\) is discrete holomorphic in \(R\),
    then the restrictions of \(\operatorname{Re} f\) and \(\operatorname{Im} f\) to
    \(\mathbb Z^2_{\mathrm{even}}\cap R\) and to
    \(\mathbb Z^2_{\mathrm{odd}}\cap R\) are harmonic with respect to
    \(\Delta_{\mathrm{diag}}\) at points
    $z_0$ in $R$ such that 
    \[
    z_0+1+i,\qquad z_0-1+i,\qquad z_0-1-i,\qquad z_0+1-i
    \]
    all belong to \(R\). 
    In short, we shall say that if \(f\) is discrete holomorphic in a region \(R\),
    then \(f\) is also harmonic in \(R\), where harmonicity is understood in the
    sense described above.

    \subsection{Existence of the function $q$}\label{s:exisrtence_q}
    Let us present a general argument that we will later specialize to prove the existence of the function $q$ introduced in \eqref{eq:q}. We want to solve the equation
    \begin{equation}\label{eq:L_w}
    X_iq(m+in)=w(m+in)
    \end{equation}
    where $w$ is a lattice function which satisfies
    \[
    \sum_{m,n\in\mathbb Z}|w(m+in)|<\infty.
    \]
    This assumption guarantees the continuity of the function
    \[
    W(x,y)=\sum_{m,n\in\mathbb Z} w(m+in)e^{-i(mx+ny)},\quad (x,y)\in\mathbb R^2.
    \]
    An explicit solution of \eqref{eq:L_w} is given by
    \[
    q(m+in)=\frac{1}{(2\pi)^2}\int_{0}^{2\pi}\int_{0}^{2\pi}\frac{W(x,y)}{\rho(x,y)}e^{i(mx+ny)}\operatorname{d}\!x\operatorname{d}\!y
    \]
    where
    \[
    \rho(x,y)=1+ie^{ix}-ie^{iy}-e^{ix+iy}.
    \]
    One can check that $|\rho(x,y)|^2=4-4\cos(x)\cos(y), x,y\in[0,2\pi]$. This proves that the integral that defines $q$ converges absolutely, for $W$ is bounded and
    \[
    \int_{0}^{2\pi}\int_{0}^{2\pi} \frac{1}{\sqrt{1-\cos(x)\cos(y)}}\operatorname{d}\!x\operatorname{d}\!y <\infty.
    \]
    Now, specializing to the case $w(m+in)=\delta_{(0,0)}$, we get $W(x,y)=1$ and  we obtain 
    \begin{equation}\label{eq:q_explicit}
    q(m+in)=\frac{1}{(2\pi)^2}\int_{0}^{2\pi}\int_{0}^{2\pi}\frac{e^{i(mx+ny)}}{1+ie^{ix}-ie^{iy}-e^{ix+iy}}\operatorname{d}\!x\operatorname{d}\!y.    
    \end{equation}
    This function satisfies \eqref{eq:q}.
    \subsection{Kernel estimates}
    The following two technical lemmas will be used throughout the paper. The function $C$ introduced below will later play the role of the Cauchy kernel of the discrete upper half-lattice.
     The first lemma is proved in \cite[Section 4]{D}, and we only recall the main steps of the argument. The second result is also mentioned there, where it is deduced from the asymptotic behavior of the discrete Green function and from earlier results in the literature. We provide a direct self-contained proof in the appendix.
        \begin{lem}\label{l:C_expansion}
            Let $q$ be the function defined in \eqref{eq:q_explicit} and set 
            \begin{equation}\label{eq:Cauchy_kernel}
            C(z)=q(-z)+iq(1-z).
            \end{equation}
        We have the following:
        \begin{enumerate}
        \item[$(i)$] $C$ is discrete holomorphic on $\mathbb Z^2_+$.
        \item[$(ii)$] For every $n\geqslant 0$,\begin{align*}
        C(m+in)&=\frac{1}{2\pi}\int_0^{\pi}e_t(m+in)\operatorname{d}\!t=\frac{1}{2\pi}\int_0^{\pi}e^{imt}\Big(\frac{\cos t}{1+\sin t}\Big)^n\operatorname{d}\!t,
        \end{align*} 
        whereas, for every $n<0$,
        \begin{align}\label{eq:kernel_negative}
        C(m+in)&=(-1)^{m+n+1}C(m+i|n|).
        \end{align}
        \end{enumerate}
        \end{lem}
    \begin{proof}
        Assertion $(i)$ easily follows by definition.  About $(ii)$ we have
        \begin{align*}
        C(m+in)&=\frac{1}{(2\pi)^2} \int_{-\pi}^{\pi}\int_{-\pi}^{\pi}  \frac{e^{-i(mx+ny)}}{1-R(x) e^{iy}}\operatorname{d}\!y\operatorname{d}\!x,
        \end{align*}
        where $R(x)=\frac{i+e^{ix}}{1+ie^{ix}}$.
        The conclusion follows by computing the inner integral in $\operatorname{d}\!y$ using the geometric series expansion of
        $(1-R(x)e^{iy})^{-1}$ in the regions
        $|R(x)|<1$ and $|R(x)|>1$. For the details we refer the reader to \cite[Section 4, equation (60)]{D}.
        \end{proof}
        
    \begin{lem}\label{l:kernel_asymptotic}
    Let $m+in\in\mathbb Z^2\setminus\{0\}$. Then,
    \begin{equation}
    \label{eq:asymptotic}
    C(m+in)=
    \frac{i}{2\pi}\left(
    \frac{1}{m+in}
    -
    (-1)^{m+n}\frac{1}{m-in}
    \right)
    +O\!\left((m^2+n^2)^{-\frac32}\right).
    \end{equation}
    The implicit constant in the remainder term does not depend on $m,n$. Notice that, more precisely, \(
    C(m+in)=0
    \)
    on the set $\big(2\mathbb Z \cup i(2\mathbb Z+1)\big)\setminus\{0\}$.
    Moreover, for every $|n|\geqslant 1$,
    \begin{equation}
    \label{eq:K_norm_p}
    \begin{aligned}
    &\|C(\cdot+in)\|_{\ell^p}
    &\leqslant cn^{\frac1p-1},
    && 1<p\leqslant\infty,
    \end{aligned}
    \end{equation}
    where we set $1/p=0$ if $p=\infty$.
    \end{lem}

    \section{Properties of  $H^p$ spaces}\label{s:Hp}
    In this section we establish the basic structural properties of the spaces
    $H^p$. The main point is the boundary norm identity
    \[
    \|f\|_{H^p}
    =
    \|f|_{\mathbb Z}\|_{\ell^p},
    \qquad 1\leqslant p<\infty.
    \]
    Its proof requires some work. We first derive a Cauchy-type reproducing formula
    in the upper half-lattice, and then show that the restriction 
    $f|_{\mathbb Z}$ reproduces $f$ through a Poisson-type kernel. The positivity and normalization of this
    kernel imply the monotonicity of the horizontal $\ell^p$-norms, and hence the
    desired norm identity.
    \begin{thm}
    For every $1\leqslant p\leqslant\infty$ the quantity
    \[
    \|f\|_{H^p}=\sup_{n\geqslant 0}\|f(\cdot+in)\|_{\ell^p}=\sup_{n\in\mathbb N} \Big(\sum_{m=-\infty}^{+\infty}|f(m+in)|^p\Big)^{\frac1p},
    \]
    with the usual modification for $p=\infty$,
    defines a norm. Moreover, the spaces $H^p$ endowed with this norm are Banach spaces.
     
    \end{thm}
    
    \begin{proof}
    It follows by the definition that 
    \begin{equation}\label{eq:bounded_functional}
    |f(m+in)|\leqslant \|f\|_{H^p},\quad m+in\in\mathbb{Z}^2_+.
    \end{equation}
    This ensures that $\|\cdot\|_{H^p}$ is a norm since $\|f\|_{H^p}=0$ implies $f\equiv 0$. Thus, $H^p$ is a normed vector space. The proof of the completeness is standard. If $\{f_k\}_{k\in\mathbb{N}}$ is a Cauchy sequence in $H^p$, then,  for every $m+in\in\mathbb{Z}^2_+$, the sequence  $\big\{ f_k(m+in)\}_{k\in\mathbb{N}}$ is a Cauchy sequence in $\mathbb{C}$ thanks to \eqref{eq:bounded_functional}.   Hence, the function $f\colon \mathbb{Z}^2_+\to\mathbb{C}$,
     \[
     f(m+in)=\lim_{k\to+\infty} f_k(m+in)
     \]
     is well-defined. Moreover, $f$ is trivially holomorphic. Let now $n,N$ in $\mathbb N$ be fixed and observe that 
     \[
     \sum_{|m|\leqslant N} |f(m+in)|^p = \lim_k \sum_{|m|\leqslant N} |f_k(m+in)|^p \leqslant \limsup_k \|f_k\|_{H^p}^p<\infty,
     \] 
     since $\{f_k\}_{k\in\mathbb N}$ is a Cauchy sequence in $H^p$, hence bounded.
     In particular, 
     \[
     \sup_{n\in\mathbb N}\sum_{m=-\infty}^{+\infty} |f(m+in)|^p \leqslant \sup_k \|f_k\|_{H^p}^p,
     \]
    that is, $f\in H^p$. It only remains to show that  $\lim_k\|f_k-f\|_{H^p}=0.$ Let $\varepsilon>0$ and choose $M\in\mathbb{N}$ such that $\|f_k-f_j\|_{H^p}^p<\varepsilon,$ for $k,j>M$. For $j\geqslant M$ and  for every $n,N$ in $\mathbb{N}$ we have that
     \[
     \sum_{|m|\leqslant N} |f(m+in)-f_j(m+in)|^p = \lim_k \sum_{|m|\leqslant N} |f_k(m+in)-f_j(m+in)|^p \leqslant \limsup_k \|f_k-f_j\|_{H^p}^p<\varepsilon,
     \]
    and this concludes the proof for $1\leqslant p<\infty$. The case $p=\infty$ follows by a similar argument.
     \end{proof}
    
    It is not apparent from the definition that $H^2$ is a Hilbert space. This will follow from the norm identity
    \[
    \|f\|^2_{H^2}=\sum_{m=-\infty}^{+\infty}|f(m)|^2,
    \]
    that will be proved in Theorem \ref{p:H^p_norm_2}.
    
    \subsection{Cauchy reproducing formula for the half-lattice}\label{s:cauchy}
    
    In this section we prove a generalization of Corollary \ref{l:cauchy_formula} for functions in $H^p, 1\leqslant p<\infty$. 
    \begin{prop}\label{p:cauchy_half_plane}
    Let $f$ in $H^p, 1\leqslant p <\infty$, and let $C$ be the kernel defined in \eqref{eq:Cauchy_kernel}. If $\operatorname{Im}(z_0)\geqslant 0$, then
    \begin{equation}\label{eq:cauchy_half_plane}
     f(z_0)=\sum_{j=-\infty}^{+\infty}f(j)C(z_0-j) .  
    \end{equation}
    If $\operatorname{Im}(z_0)<0$, then
    \begin{equation*} 
    \sum_{j=-\infty}^{+\infty}f(j)C(z_0-j)=0.   
    \end{equation*}
    \end{prop}
        
    \begin{proof}
     Given $M,N\in\mathbb{N}$, consider the rectangle $R_{M,N}$ having as vertices the points $-M, M+iN,-M+iN$ and $-M$. If $\operatorname{Im}z_0\geqslant0$, for $M,N$ big enough so that the square associated to $z_0$ is contained in $R_{M,N}$, equation \eqref{eq:cauchy} gives
    \[
    f(z_0)=\int_{\partial R_{M,N}} f:\big((1+i)q_{z_0}+iq_{z_0}'-\frac{1-i}{4}q_{z_0}''\big) .  
    \]
    This line integral can be decomposed into the sum of four terms, each one corresponding to an edge of $R_{M,N}$. From the term corresponding to the edge of vertices $-M$ and $M$, excluding the two endpoints,  we get
    \begin{align}\label{eq:lower_edge}
        \sum_{j=-M+1}^{M-1} f(j) q_{z_0,A}(j),
    \end{align}
    where
    \begin{align*}
        q_{z_0,A}(j)&:=(1+i)\frac{q_{z_0}(j-1)+2q_{z_0}(j)+q_{z_0}(j+1)}{4} +i\frac{q_{z_0}(j+1)-q_{z_0}(j-1)}{2} +\\
        &\qquad -\frac{1-i}{4}\big( (q_{z_0}(j+1)-q_{z_0}(j)) -(q_{z_0}(j)-q_{z_0}(j-1))\big)\\
        &=q_{z_0}(j)+iq_{z_0}(j+1)\\
        &=C(z_0-j).
    \end{align*}
    Similarly, from the term corresponding to the edge with vertices $-M+iN$ and $M+iN$, excluding the two end-points,  we get
    \begin{align}\label{eq:upper_edge}
        \sum_{j=-M+1}^{M-1} f(j+iN) q_{z_0,B}(j),
    \end{align}
    where
    \begin{align*}
         q_{z_0,B}(j):=
        -q_{z_0}(j+iN)-iq_{z_0}(j+1+iN)=-C(z_0-j-iN).
    \end{align*}
    The terms corresponding to the vertical edges (end-points excluded) of $R_{M,N}$ are linear combinations of the form
    \begin{align}\label{eq:vertical_edge}
        \sum_{j=1}^{N-1} f(\pm M+ij) \big(\alpha q_{z_0}(\pm M+ij-i)+\beta q_{z_0}(\pm M+ij)+\gamma q_{z_0}(\pm M+ij+1)\big),
    \end{align}
    where $\alpha,\beta,\gamma\in\mathbb{C}$.  We now want to take the limit as $M\to+\infty$ in \eqref{eq:lower_edge}, \eqref{eq:upper_edge} and \eqref{eq:vertical_edge}.

    Notice that, if $z_0=x_0+iy_0$, by H\"older's inequality and Lemma \ref{l:kernel_asymptotic}, for $1\leqslant p<\infty$, we get
    \begin{align}\label{eq:upper_edge_2}
    \begin{split}
     \Big|\sum_{j=-\infty}^{+\infty}f(j+iN)q_{z_0,B}(j)\Big|&= \lim_{M\to+\infty}\Big|\sum_{j=-M+1}^{M-1}f(j+iN)C(z_0-j-iN)\Big|\\
     &\leqslant c \cdot (N-y_0)^{\frac 1{p'}-1}\|f\|_{H^p},
     \end{split}
    \end{align}
    where $p'$ is the conjugate exponent of $p$.
    Similarly, Lemma \ref{l:kernel_asymptotic} guarantees that 
    \begin{equation}\label{eq:lower_edge_limit}
    \sum_{j=-\infty}^{+\infty}f(j)q_{z_0,A}(j)=\lim_{M\to+\infty} \sum_{j=-M+1}^{M-1} f(j)C(z_0-j)
    \end{equation}
    is well-defined as well if $y_0>0$. If $y_0=0$ we need to observe that, by Lemma \ref{l:C_expansion},
    \begin{equation*}
    C(m)=
    \begin{cases}
    \frac12, & m=0,\\[4pt]
    0, & m\in 2\mathbb Z\setminus\{0\},\\[4pt]
    \dfrac{i}{\pi m}, & m\in 2\mathbb Z+1.
    \end{cases}
    \end{equation*}
    This is still enough to guarantee that \eqref{eq:lower_edge_limit} is well-defined. About the vertical contributions in \eqref{eq:vertical_edge}, notice that they tend to $0$ as $M\to+\infty$ since they are given by finite sums and both $f$ and $q$ are null sequences along horizontal lines.  Similarly, for fixed $N$, the contributions of the four vertices of the rectangle $R_{M,N}$ tend to $0$ as $M\to+\infty$. Hence, so far we proved the identity
    \[
    f(z_0) =\sum_{j=-\infty}^{+\infty} f(j)C(z_0-j)-\sum_{j=-\infty}^{+\infty} f(j+iN)C(z_0-j-iN), \qquad N> \operatorname{Im}(z_0).
    \] 
    To conclude it is now sufficient to take the limit as $N\to+\infty$ and consider the estimate \eqref{eq:upper_edge_2}.

    Finally, notice that if $\operatorname{Im}(z_0)<0$, then for every $M,N$
    \[
    \int_{\partial R_{M,N}} f:\big((1+i)q_{z_0}+iq_{z_0}'-\frac{1-i}{4}q_{z_0}''\big) =0,
    \]
    by the holomorphicity of $f$, the definition of $q_{z_0}$ and Propositions \ref{p:integrals_identities} and \ref{p:integrals_q}. Therefore, with the same argument as above, it can be shown that
    \[
    \sum_{j=-\infty}^{\infty} f(j)C(z_0-j)=0, \qquad \operatorname{Im}(z_0)<0.
    \]
    \end{proof}
    
    \subsection{Poisson reproducing formula for the half-lattice}
    
      From this point on, we refer to the kernel $C$ as the \emph{Cauchy kernel} in our setting, as it plays a role analogous to that of its classical continuous counterpart. The goal of this section is to show that the reproducing formula \eqref{eq:cauchy_half_plane} remains valid when the Cauchy kernel $C$ is replaced, essentially, by its real part. The point of passing to the real part is that it produces a positive,
    $\ell^1$-normalized Poisson-type kernel. This is the ingredient that will later
    imply the monotonicity of the horizontal $\ell^p$-norms and hence the boundary
    norm identity.

      By the elementary symmetry
    \begin{equation}\label{Eq:symmetryphi}
    \frac{\cos (\pi-t)}{1+\sin (\pi-t)}=-\frac{\cos t}{1+\sin t}, \qquad -\frac{\pi}{2}<t<\frac{3}{2}\pi,
    \end{equation}
    
    it follows that
    \begin{align*}
      C(m+in)&=\frac{1}{2\pi}\int_0^\pi e_t(m+in)\operatorname{d}\!t=\frac{1}{2\pi}\int_0^{\pi} e^{imt}\Big(\frac{\cos t}{1+\sin t}\Big)^n\operatorname{d}\!t\\
      &=\frac{1}{2\pi}\int_0^{\frac\pi 2}(e^{imt}+(-1)^{m+n}e^{-imt})\Big(\frac{\cos t}{1+\sin t}\Big)^n\operatorname{d}\!t\\
      &=\frac{1}{2\pi}\int_0^{\frac{\pi}{2}}(e_t(m+in)+(-1)^{m+n}e_t(-m+in))\operatorname{d}\!t.
      \end{align*}
    In particular,
    \begin{equation}\label{eq:poisson_kernel}
     \operatorname{Re}(C(m+in))=\begin{cases}
          \dfrac{1}{\pi}\displaystyle\int_0^{\frac\pi 2} \cos(mt)\Big(\dfrac{\cos t}{1+\sin t}\Big)^n\,\mathrm{d}t,
          & m+n\in 2\mathbb Z,\\
          0, & m+n\in 2\mathbb Z+1,
      \end{cases}  
    \end{equation}
    and 
    \begin{equation}\label{eq:poisson_conjugate}
    \operatorname{Im}(C(m+in))=\begin{cases}
        0, & m+n\in 2\mathbb Z,\\
        \dfrac{1}{\pi}\displaystyle\int_0^{\frac\pi 2} \sin(mt)\Big(\dfrac{\cos t}{1+\sin t}\Big)^n\,\mathrm{d}t,
            & m+n\in 2\mathbb Z+1.
    \end{cases}    
    \end{equation}
    Moreover, by Lemma \ref{l:kernel_asymptotic}, for $n\geqslant 1$ we have
    \begin{equation*}
     \operatorname{Re}(C(m+in))=\frac{1}{\pi} \frac{n}{m^2+n^2}+O\Big(\frac{1}{(m^2+n^2)^{\frac 32}}\Big),\qquad m+n\in 2\mathbb Z,
    \end{equation*}
    and
    \begin{equation*}
    \operatorname{Im}(C(m+in))= \frac{1}{\pi}\frac{m}{m^2+n^2}+O\Big(\frac{1}{(m^2+n^2)^{\frac 32}}\Big), \qquad m+n\in 2\mathbb Z+1.
    \end{equation*}
    The real and imaginary parts of \(C\) will play the role of the Poisson and
    conjugate Poisson kernels in the sequel. Indeed, we have the following result.
    
    \begin{prop} \label{p:reproducing_poisson}
    Set 
    \[
    P_n(m)=2\operatorname{Re}(C(m+in)),\qquad  m+in\in\mathbb Z^2_+.
    \]
    For every $f\in H^p$, $1\leqslant p<\infty$, we have
    \[
    f(m+in)= \sum_{j=-\infty}^{+\infty}f(j)P_n(m-j), \qquad m+in\in\mathbb{Z}^2_+.
    \]
    \end{prop}
    
    To prove this Poisson reproducing formula we need the following  uniqueness principle for functions that are harmonic in the sense  of Section \ref{s:harmonicity}. 
    \begin{lem}\label{l:uniqueness_harmonic}
    Let $U$ be a harmonic  function in $\mathbb Z^2_+$, in the sense of Section \ref{s:harmonicity}. Assume that
    $
    U|_{\mathbb Z}\equiv 0$ and $ \lim_{|z|\to\infty}|U(z)|=0$. Then $U\equiv 0$ in $\mathbb Z^2_+$.  
    \end{lem}
    
    \begin{proof}
    By definition of discrete harmonicity, it is straightforward to see that, given the square~$Q_M$ of vertices $M, M+2iM, -M+2iM, -M$, it holds
    \[
    \max_{Q_M}|U|=\max_{\partial Q_{M}}|U|.
    \]
    Let $\varepsilon>0$ be fixed and let  $z_0$ be in $\mathbb Z^2_+$. The hypothesis guarantees that we can choose $M>0$ large enough so that $z_0\in Q_M$ and $|U(z)|<\varepsilon$ for every $|z|\geqslant M$. In particular,
    \[
    |U(z_0)|\leqslant \max_{Q_M} |U|=\max_{\partial Q_M}|U| \leqslant \varepsilon,
    \]
    since by assumption $U(m)=0$ for every real $m$ and every other point $z$ in $\partial Q_M$  satisfies $|z|\geqslant M$. Since this holds for every $\varepsilon>0$ and $z_0\in\mathbb{Z}^2_+$, we conclude that $U\equiv 0$ on $\mathbb{Z}^2_+$.
    \end{proof}
    \noindent Let us now prove the proposition.
    \begin{proof}[Proof of Proposition \ref{p:reproducing_poisson}]  From \eqref{eq:poisson_kernel} we get $P_0=\delta_0$. Hence, the identity
        \[
        f(m)=\sum_{j=-\infty}^{+\infty}f(j)P_0(m-j), \qquad m\in\mathbb{Z},
        \]
     is trivially true.    
     Define 
     \[
     G(m+in)=\sum_{j=-\infty}^{+\infty}f(j)P_n(m-j), \qquad m+in\in\mathbb{Z}^2_+.
     \]
     Such function is well-defined and harmonic in $\mathbb Z^2_+$. The harmonicity is guaranteed by the harmonicity of $P_n(m)$, whereas the convergence of the series follows since, if $1\leqslant p<\infty$,
     \[
     \sum_{j=-\infty}^{+\infty}|f(j)P_n(m-j)|\leqslant \|f\|_{H^p}\|P_n\|_{\ell^{p'}}\leqslant c n^{\frac1{p'}-1}\|f\|_{H^p}.
     \]
     Now, the function 
     \[
     U(m+in)=f(m+in)-G(m+in), \qquad m+in\in \mathbb Z^2_+,
     \]
     vanishes on $\mathbb Z$ and 
     is harmonic on $\mathbb Z^2_+$. We show that $\lim_{|z|\to\infty}|U(z)|=0$, so that we can apply Lemma \ref{l:uniqueness_harmonic}. By Proposition \ref{p:cauchy_half_plane} and Lemma \ref{l:C_expansion}, we have that
      \begin{align}\label{eq:U_vertically}
      \begin{split}
          |U(m+in)|&\leqslant \sum_{j=-\infty}^{+\infty}|f(j)||C(m+in-j)-P_n(m-j)|\\
    &=\sum_{j=-\infty}^{+\infty}|f(j)||C(m+in-j)|    \leqslant c  n^{\frac 1{p'}-1}\|f\|_{H^p}, 
    \end{split}
     \end{align}
     uniformly for $m\in\mathbb{Z}$, where $p'$ is the conjugate exponent ot $p$. Let now $\varepsilon>0$ be fixed. Let $N_\varepsilon$ be such that 
    \[
    \sup_{m\in\mathbb{Z}}|U(m+in)|\leqslant c n^{\frac 1{p'}-1}\|f\|_{H^p}<\varepsilon,\quad n\geqslant N_\varepsilon.
    \]
     If $n< N_\varepsilon$, notice that
    \begin{equation}\label{eq:U_horizontally}
     \lim_{m\to+\infty}U(m+in)=0
     \end{equation}
     since both $f$ and $G$ are null sequences horizontally. Thus, we can find $M_\varepsilon$ such that 
     \[
     |U(m+in)|\leqslant \varepsilon,\qquad  |m|\geqslant M_\varepsilon,\quad 0\leqslant n<N_\varepsilon.
     \]
     In conclusion, by \eqref{eq:U_vertically} and \eqref{eq:U_horizontally}
    \[
    |U(m+in)|\leqslant \varepsilon,\qquad |m+in|\geqslant \sqrt{M^2_\varepsilon+N^2_\varepsilon}.
    \]
    Hence, we can apply Lemma \ref{l:uniqueness_harmonic} to $U$ and conclude that 
    \[
    f(m+in)=\sum_{j=-\infty}^{+\infty} f(j) P_n(m-j), \quad m+in\in\mathbb Z^2_+,
    \]
    as we wished to show.
    \end{proof}
    The reproducing formula we just proved will be fundamental in proving that 
    \begin{equation}\label{eq:H^p_norm}
    \|f\|^p_{H^p}=\sum_{m=-\infty}^{+\infty}|f(m)|^p.
    \end{equation}
    However, we still need two preliminary results.

    \begin{lem}\label{l:Poisson_non_negative}
    For every $m+i n\in\mathbb Z^2_+$ we have
     $P_n(m)\geqslant 0$. 
    \end{lem}
    
    \begin{proof}
    The proof follows from a minimum principle and a standard argument for real harmonic functions, that are positive on the boundary. For $N\geqslant 0$, let $Q_N$ be the square in the upper half-lattice of vertices $-N, N, N+2iN, -N+2iN$. Since $P_n(m)=2\operatorname{Re}(C(m+in))$ is a real-valued harmonic function in the sense of \ref{s:harmonicity}, it is easy to see that 
    \[
    M_N:=\min_{m+in \in Q_N}P_n(m)=\min_{m+in \in \partial Q_N}P_n(m).
    \]
    Fix $m+in\in\mathbb{Z}^2_+$, and take $N_0$ big enough so that $m+in\in Q_{N_0}.$ By definition of minimum, 
    \begin{equation}\label{eq:harmonic+}
        P_n(m)\geqslant M_N, \qquad N\geqslant N_0.
    \end{equation}
    Since on the boundary $P_0=\delta_0$ is positive, if for some $N$ the minimum $M_N$ is attained on the horizontal edge $[-N,N]$, then $P_n(m)\geqslant 0$. If else for every $N$ the minimum $M_N$ is attained on the other three edges of $Q_N,$ since $\lim_{m^2+n^2\to+\infty}P_n(m)=0$ by Lemma \ref{l:C_expansion}, it follows that $M_N\to 0$ as $N\to+\infty$. Then, taking the limit as $N\to+\infty$ in \eqref{eq:harmonic+}, we conclude that $P_n(m)\geqslant 0.$
    \end{proof}

    \begin{lem}\label{l:P_n_norm_1} For every $n\in\mathbb N$ we have
    \[
    \|P_n\|_{\ell^1}=\sum_{m=-\infty}^{+\infty}P_n(m)=1.
    \]
    \end{lem}
    
    \begin{proof}
    The first identity follows by the previous lemma. For the second identity we proceed as follows.
     Let $\varphi_n(t)=\Big(\frac{\cos|t|}{1+\sin|t|}\Big)^n$. Then, 
    \begin{align}\label{eq:Poisson_coefficient}
     \frac{1}{\sqrt{2\pi}}(\mathcal F^{-1}\varphi_n)(m)=\frac{1}{2\pi}\int_{-\pi}^{\pi}\varphi_n(t) e^{imt}\operatorname{d}\!t&=\frac{1}{\pi}\int_0^\pi \cos(mt)\Big(\frac{\cos t}{1+\sin t }\Big)^n\operatorname{d}\! t=P_n(m).   
    \end{align}
    Hence,
    \begin{align*}
     \sum_{m=-\infty}^{+\infty}P_n(m)=\frac{1}{\sqrt{2\pi}} \sum_{m=-\infty}^{+\infty}(\mathcal F^{-1}\varphi_n)(m)= \varphi_n(0)=1.   \end{align*}

    The second-last identity holds true thanks to the pointwise convergence of the Fourier series of $\varphi_n$ at $t=0$.
    \end{proof}

    We conclude the section by proving \eqref{eq:H^p_norm}.
    \begin{prop}\label{p:H^p_norm_2} 
    Let $1\leqslant p<\infty$ and $f$ in $H^p$. For every $n$ in $\mathbb N$ we have
    \[
    \sum_{m=-\infty}^{+\infty} |f(m+i(n+1))|^p\leqslant \sum_{m=-\infty}^{+\infty}|f(m+in)|^p.
    \]
    Hence,
    \[
    \|f\|_{H^p}=\Big(\sum_{m=-\infty}^{+\infty}|f(m)|^p\Big)^{\frac 1p}.
    \]
    \end{prop}
    
    \begin{proof}
      Set $f_n(z)=f(z+in)$. If $f\in H^p$, then $f_n\in H^p$ as well. Moreover, by Proposition \ref{p:reproducing_poisson} and Lemma \ref{l:Poisson_non_negative}, we get
    \begin{align*}
    \Big(\sum_{m=-\infty}^{+\infty}|f(m+i(n+1))|^p\Big)^{\frac1p}&=\Big(\sum_{m=-\infty}^{+\infty}|f_n(m+i)|^p\Big)^{\frac1p}=\Big(\sum_{m=-\infty}^{+\infty}|(f_n\ast P_1)(m)|^p\Big)^{\frac1p}   \\
    &\leqslant \|f_n\|_{\ell^p}\|P_1\|_{\ell^1}=\Big(\sum_{m=-\infty}^{+\infty}|f(m+in)|^p\Big)^{\frac1p}.   
    \end{align*}
    Thus, the norms $\|f_n\|_{\ell^p}$ are decreasing in $n$. Hence, the conclusion follows.
    \end{proof}
    
    In particular, we have proved the following result. The statement is very basic and expected, but the proof did require a lot of work. 
    
    \begin{thm}\label{t:Hp_boundary_banach}
    The $H^p$ spaces, $1\leqslant p<\infty$,
    are Banach spaces with the norm
    \[
    \|f\|_{H^p}
    =
    \Big(\sum_{m=-\infty}^{+\infty}|f(m)|^p\Big)^{\frac 1p}.
    \]
    \end{thm}
    \begin{rmk}\label{Rmk:kernelsHp}
        Take the kernel $C$ considered in Lemmas \ref{l:C_expansion} and \ref{l:kernel_asymptotic}. As a straightforward consequence of those lemmas, we obtain examples of functions in $H^p$. For $w\in\mathbb{Z}^2_+,$ the function
        \[K_w(z)=C(z-\overline{w}), \qquad z\in\mathbb{Z}^2_+,\]
        belongs to $H^p$, for every $p>1.$
    \end{rmk}
    \begin{rmk}
    From now on, we refer to $P_n$ as the \emph{Poisson kernel} of the discrete upper half-lattice. This is motivated by the following Dirichlet problem in the
    discrete upper half-lattice. Given boundary data \(g:\mathbb Z\to\mathbb C\), we
    look for a function \(u:\mathbb Z^2_+\to\mathbb C\) such that
    \[
    \begin{cases}
    \Delta_{\mathrm{diag}}u(z)=0, & z\in\mathbb Z^2_+,\ \operatorname{Im}z>0,\\[4pt]
    u(m)=g(m), & m\in\mathbb Z,\\[4pt]
    \displaystyle \lim_{\substack{|z|\to\infty\\ z\in\mathbb Z^2_+}} |u(z)|=0.
    \end{cases}
    \]
    For suitable data, for instance \(g\in\ell^p(\mathbb Z)\), \(1\leqslant p<\infty\),
    the solution is given by
    \[
    u(m+in)
    =
    \sum_{j=-\infty}^{+\infty} g(j)P_n(m-j),
    \qquad m+in\in\mathbb Z^2_+.
    \]
    Indeed, \(P_n=2\operatorname{Re}C(\cdot+in)\) is harmonic in the upper
    half-lattice, \(P_0=\delta_0\) on the boundary, and the estimates for \(P_n\)
    ensure the required convergence and decay at infinity. The uniqueness of this
    solution follows from Lemma~\ref{l:uniqueness_harmonic}.
    \end{rmk} 
    
    \section{The space $H^2$ and the Szeg\H o projection} \label{s:H2}
    Thanks to Proposition \ref{p:H^p_norm_2}, the space $H^2$ is a reproducing kernel Hilbert space with respect to the inner product
    \[
    \langle f,g\rangle_{H^2}=\sum_{m=-\infty}^{+\infty}f(m)\overline{g(m)}.
    \]
    
    For each $w=u+iv\in\mathbb Z^2_+$, its reproducing kernel $K_{w}$ in $z=m+in$ is given by \begin{equation}\label{eq:kernel}
       K_{w}(z)=C(z-\overline{w})
       =\frac{1}{2\pi}\int_0^{\pi} e_t((m-u)+i(n+v))\operatorname{d}\!t.
    \end{equation}
    
    Indeed, that $K_w\in H^2$ was already discussed in Remark \ref{Rmk:kernelsHp},
    and thanks to Proposition \ref{p:cauchy_half_plane}, for every $w$ in $\mathbb Z^2_+$, we have
    \[
    \langle f, K_{w}\rangle_{H^2}= \sum_{m=-\infty}^{+\infty}f(m) \overline{K_{w}(m)}=\sum_{m=-\infty}^{+\infty} f(m)C(w-m)=f(w).
    \]
    We now identify $H^2$ with a closed subspace of $\ell^2(\mathbb Z)$. More generally, for every $1\leqslant p<\infty$, let us consider the restriction operator $J: H^p\to \ell^p(\mathbb Z)$ defined by 
    \[
    f \in H^p \mapsto Jf=(f(m))_{m\in\mathbb Z} \in \ell^p(\mathbb Z).
    \]
    Proposition \ref{p:H^p_norm_2}
    guarantees that $J$ acts as an isometry from $H^p$ to $\ell^p(\mathbb Z)$. We can rewrite the reproducing formula \eqref{eq:cauchy_half_plane} on the boundary in terms of the classical convolution product $\ast$ on~$\mathbb Z$:
    \begin{equation}\label{E:reproducingformulaconvolution}
     f(m)=\sum_{j=-\infty}^{+\infty}f(j)C(m-j)=f\ast C(m), \qquad m\in\mathbb{Z},f\in H^p.
    \end{equation}
    From Proposition \ref{p:cauchy_half_plane} we deduce a Paley--Wiener type characterization of the range~$J(H^2)$. 
    \begin{thm}\label{t:H^2_PW}
    The Fourier transform  $\mathcal F$ maps \(J(H^2)\) onto
    \[
    \left\{
    g\in L^2([-\pi,\pi])\colon \operatorname{supp}g\subseteq[0,\pi]
    \right\}.
    \]
    Moreover, for every \(f\in H^2\),
    \begin{equation}\label{eq:F_isometry}
        \|\mathcal F(Jf)\|_{L^2([-\pi,\pi])}
    =\|f\|_{H^2}.
    \end{equation}
    Finally, the orthogonal projection $P\colon\ell^2\to J(H^2)$, that is the Szeg\H o projection, is given by 
    \begin{equation}\label{E:Porthoproj}
        Pf=f\ast C=\mathcal{F}^{-1}\big(\chi_{[0,\pi]}\mathcal{F}f\big),
    \end{equation}
    where $\chi_{[0,\pi]}$ denotes the characteristic function of the interval $[0,\pi]$.
    \end{thm}
    We point out that this Paley--Wiener type characterization is not new; it was
    already proved by Zeilberger and Dym in \cite[Theorem W3]{ZD}. However, their approach
    establishes the spectral characterization first and then uses it to derive the
    reproducing kernel Hilbert space structure of $H^2$. In contrast, our approach
    shows that the reproducing kernel Hilbert space structure of \(H^2\) can be
    obtained intrinsically from discrete holomorphicity and the Cauchy formula,
    without invoking a spectral model. The Paley--Wiener characterization is then
    recovered as an immediate consequence of the Cauchy reproducing formula on the
    boundary.
    
    \begin{proof}[Proof of Theorem \ref{t:H^2_PW}]
    Let $f$ be a function in $H^2$. Since $C\in\ell^2(\mathbb{Z})$, the product $\mathcal{F}(Jf)\cdot \mathcal{F}C$ belongs to $L^1$, and computing its Fourier coefficients one obtains the identity
    \[
    \mathcal{F}^{-1}\big(\mathcal{F}(Jf)\cdot \mathcal{F}C\big)(m)=\frac{1}{\sqrt{2\pi}}\big(Jf\ast C\big)(m), \qquad m\in\mathbb{Z}.
    \]
    Then, by \eqref{E:reproducingformulaconvolution} we get
    \[
    \mathcal F(Jf)= \mathcal F(Jf\ast C)=\sqrt{2\pi}(\mathcal F (Jf))(\mathcal FC).
    \]
    Recall that
    \begin{equation}\label{E:C=F^{-1}}
    C(m)=\frac{1}{2\pi}\int_0^{\pi}e^{imt}\operatorname{d}\!t=\frac{1}{\sqrt{2\pi}}\big(\mathcal{F}^{-1}\chi_{[0,\pi]}\big)(m), \qquad m\in\mathbb{Z}.
    \end{equation}
    Then, 
    \[
    \operatorname{supp}(\mathcal F(Jf)) = \operatorname{supp}\!\big(\mathcal F (Jf)\chi_{[0,\pi]}\big)\subseteq [0,\pi],\qquad f\in H^2.
    \]
    Conversely, let \(g\in L^2([-\pi,\pi])\) be supported in \([0,\pi]\), and define
    \begin{equation}\label{eq:f_representation_spectral}
    f(m+in)
    =\frac{1}{\sqrt{2\pi}}\int_0^\pi
    g(t)e_t(m+in)
    \operatorname{d}\!t=
    \frac{1}{\sqrt{2\pi}}\int_0^\pi
    g(t)e^{imt}
    \left(\frac{\cos t}{1+\sin t}\right)^n
    \operatorname{d}\!t.
    \end{equation}
    This function $f$ is discrete holomorphic in $\mathbb Z^2_+$.  
    Moreover, by Plancherel,
    \[
    \|f(\cdot+in)\|_{\ell^2}^2
    =
    \int_0^\pi
    |g(t)|^2
    \left|
    \frac{\cos t}{1+\sin t}
    \right|^{2n}
    \operatorname{d}\!t
    \leqslant
    \int_0^\pi |g(t)|^2\operatorname{d}\!t ,\qquad n\geqslant 0.
    \]
    Thus $f\in H^2$ and $\mathcal F(Jf)=g$.
    This proves the reverse inclusion. The norm identity \eqref{eq:F_isometry} is just Plancherel's formula. Finally, we prove that the orthogonal projection $P$ is given by \eqref{E:Porthoproj}. The second identity follows at once from \eqref{E:C=F^{-1}}. That the range $P(\ell^2)$ coincides with $J(H^2)$ is a consequence of the Paley--Wiener characterization, and that $Pf=f$ for $f$ in $J(H^2)$ follows from the reproducing formula \eqref{E:reproducingformulaconvolution}. Finally, since $\mathcal{F}^*=\mathcal{F}^{-1}$, it is straightforward that $P^*=P$.
    \end{proof}

    \begin{rmk}  
    We call the orthogonal projection \(
    P:\ell^2(\mathbb Z)\to J(H^2)
    \) the \emph{Szeg\H o projection} of~$H^2$. We can rewrite the convolution on the boundary with the Cauchy kernel as
    \[
    Pf(m)
    =
    \sum_{j=-\infty}^{+\infty}f(m-j)C(j)=\frac{1}{2}\Big(\sum_{j=-\infty}^{+\infty}f(m-j)P_0(j)+i\sum_{j=-\infty}^{+\infty}f(m-j)Q_0(j)\Big),
    \]
    where $Q_0$ is the \emph{conjugate Poisson kernel}, i.e., $Q_n(m)=2\operatorname{Im}(C(m+in))$. We denote
    \[
    \mathcal H_{\operatorname{odd}}f(m)=f\ast Q_0(m)=\frac{2}{\pi}
    \sum_{j\in 2\mathbb Z+1}
    \frac{f(m-j)}{j}, \qquad m\in\mathbb Z.
    \]
    The second equality follows from \eqref{eq:poisson_conjugate}. Notice that $\mathcal H_{\operatorname{odd}}$  is a discrete Hilbert transform-type operator.
    Since $P_0=\delta_0$, we have the discrete analogue of the classical identity
    \begin{equation}\label{Eq:P_Hilbert}    
    P=\frac{1}{2}(\operatorname{ Id}+i\mathcal H_{\operatorname{odd}}).
    \end{equation}
        In particular, for $f\in J(H^2)$, by \eqref{Eq:P_Hilbert}, we obtain
    \begin{equation}\label{eq_poisson_convoluzione}
    f(m)=Pf(m)=\frac{2i}{\pi}\sum_{j\in2\mathbb Z+1}\frac{f(m-j)}{j}.
    \end{equation}
    This identity shows that, on $J(H^2)$, the values on even integers determine the values on odd integers, and vice versa. Therefore $\mathbb Z_{\mathrm{even}}$ and $\mathbb Z_{\mathrm{odd}}$ are uniqueness sets for $H^2$. Namely, if
    $f|_{\mathbb Z_{\operatorname{even}}}=0$ or 
    $f|_{\mathbb Z_{\operatorname{odd}}}=0$
    then $f\equiv 0$ on $\mathbb Z$ and, consequently, by Theorem \ref{t:Hp_boundary_banach}, $f\equiv 0$ on $\mathbb Z^2_+$. We can actually generalize the results to every horizontal line. Indeed, if $f\in H^2$
    and $n\geqslant 0$, then the vertical translate
    $
    f_n(z)=f(z+in)
    $
    also belongs to $H^2$, and its boundary values are precisely
    $f_n(m)=f(m+in)$. Applying identity \eqref{eq_poisson_convoluzione} to $f_n$, we obtain
    \begin{equation}\label{eq:f_even_odd}
    f(m+in)
    =
    \frac{2i}{\pi}
    \sum_{j\in2\mathbb Z+1}
    \frac{f(m-j+in)}{j}.
    \end{equation}
    Thus, for every $n\geqslant0$, the values of $f$ on $ \Lambda^{\operatorname{even}}_n=\{2m+in:m\in\mathbb Z\}$ determine the values of $f$ on $\Lambda^{\operatorname{odd}}_n=\{2m+1+in:m\in\mathbb Z\},$
    and vice versa. Thus, if either $f|_{\Lambda^{\operatorname{even}}_n}=0$ or $f|_{\Lambda^{\operatorname{odd}}_n}=0$, then $f(m+in)=0$ for every $m\in\mathbb Z$. Applying the boundary norm identity to the vertical translate $f_n(z)=f(z+in)$, we obtain $f_n\equiv0$. Hence $f(m+ik)=0$ for every $m\in\mathbb Z$ and every $k\geqslant n$. However, holomorphicity also implies that $f(m+ik)=0$ also for every $0\leqslant k<n$. Indeed, if $f$ vanishes identically on the horizontal line of height $n$, then, for every $m\in\mathbb Z$, by holomorphicity, we would have
     \begin{align*}
     0&=f(m+1+in)-f(m+i(n-1))+i(f(m+in)-f(m+1+i(n-1)))\\
     &=-f(m+i(n-1))-if(m+1+i(n-1)).
     \end{align*}
    Hence, $|f(m+i(n-1))|=|f(m+1+i(n-1))|$ and the sequence $m\mapsto |f(m+i(n-1))|$ is constant. Since $f(\cdot+i(n-1))\in\ell^2(\mathbb Z)$, this constant must be zero. Iterating the argument we obtain that $f\equiv 0$ in $\mathbb Z^2_+$. We proved the following.
    \begin{prop}\label{p:uniqueness}
    For every fixed $n\geqslant 0$, the sets $\Lambda_{n}^{\operatorname{even}}=\{2m+in\in\mathbb Z^2:m\in\mathbb Z\}$ and $\Lambda_{n}^{\operatorname{odd}}=\{2m+1+in\in\mathbb Z^2:m\in\mathbb Z\}$ are uniqueness sets for $H^2$. Namely, given $f$ in $H^2$, if either  
    \(
    f|_{\Lambda_n^{\operatorname{even}}}\equiv 0\) or \(
    f|_{\Lambda_n^{\operatorname{odd}}}\equiv 0,
    \) 
    then \(f\equiv 0\) on $\mathbb Z^2_+$.
    \end{prop}
    \noindent We will come back to this topic in Section \ref{s:sampling}.
    \end{rmk}
    
    We give some details on the Hilbert space structure of $H^2$. By Theorem \ref{t:H^2_PW} it is immediate to explicitly compute an orthonormal basis. 
    \begin{prop}\label{T:ONB}
    
    For $k\in\mathbb{Z}$, we define the function
    \[
    E_k(m+in)=\frac{\sqrt 2}{2\pi}\int_0^\pi e^{2kit}e_t(m+in)\operatorname{d}\!t, \qquad m+in\in\mathbb{Z}^2_+.\]
    Then, the set $\{E_k\}_{k\in\mathbb{Z}}$ is a complete orthonormal basis for $H^2$.
    \end{prop}
    \begin{proof} The proof follows from Theorem \ref{t:H^2_PW}, the fact that $\mathcal{F}JE_k=\frac{1}{\sqrt\pi}e^{2ki(\cdot)}\chi_{[0,\pi]}$ and that $\{\frac{1}{\sqrt\pi} e^{2ki(\cdot)}\chi_{[0,\pi]}\}_{k\in\mathbb{Z}}$ is an orthonormal basis of $\chi_{[0,\pi]}L^2([-\pi,\pi])$.
    \end{proof}

    \begin{rmk}
        In the previous proposition we could have found a different basis of $H^2$ by starting from the exponentials of odd frequency $\{\frac{1}{\sqrt{\pi}}e^{i(2k+1)(\cdot)}\}_{\mathbb Z}$ since they still constitute an orthonormal basis for $\chi_{[0,\pi]}L^2([-\pi,\pi])$. Thus, any function $f$ in $H^2$ can be recovered either by only considering exponentials of even or odd frequency. This is of course coherent with \eqref{eq_poisson_convoluzione} and Proposition \ref{p:sampling}.
    \end{rmk}

    Associated to the orthonormal basis $\{E_k\}_{k\in\mathbb{Z}}$, there is a natural shift operator. 
    \begin{thm}\label{T:shift}
    For $f\in H^2$, set $S_2f(m+in):=f(m+2+in)$. Then $S_2$ defines an isometry on $H^2$ and satisfies $S_2E_k=E_{k+1}$ for every $k\in\mathbb Z$. Moreover, $S_2$ is unitarily equivalent to the bilateral shift $\Sigma:L^2([-\pi,\pi])\to L^2([-\pi,\pi])$, defined by $\Sigma g(t)=e^{it}g(t)$.
    \end{thm}
    \begin{proof}
    The fact that $S_2$ is an isometry on $H^2$ and shifts the elements of the basis $\{E_k\}_{k\in\mathbb Z}$ is straightforward.
    It is also not difficult to check that $S_2=(V\mathcal F)^{-1}\Sigma V\mathcal F$ where 
       \[V\colon \chi_{[0,\pi]}L^2(\mathbb{T})\to L^2(\mathbb{T}), \qquad  V(\chi_{[0,\pi]}f)(e^{i\theta})= \frac{1}{\sqrt{2}}f(e^{i\frac\theta2}), \,\theta\in[0,2\pi].\]
    In particular, $V\mathcal F$ is a unitary operator.
    \end{proof}
    
    \begin{rmk}
    Theorem \ref{T:shift} shows that the two-step horizontal translation $
    S_2$ is a natural bilateral shift on $H^2$. In contrast, the one-step horizontal
    translation
    \(
    S_1f(m+in)=f(m+1+in)
    \)
    is unitary on $H^2$, but it is not a bilateral shift.
    
    Indeed, assume by contradiction that $S_1$ is a bilateral shift. Then, by definition, there
    exists a non-zero wandering subspace \(V\subseteq H^2\) such that 
    \(
    S_1^nV\perp V
    \) for every $n\in\mathbb Z\setminus\{0\}$, and
    \[
    H^2=\bigoplus_{n\in\mathbb Z}S_1^nV.
    \]
    In particular, taking any non-zero \(f\in V\), we would have
    \[
    \langle f,S_1^nf\rangle_{H^2}=0,
    \qquad n\in\mathbb Z\setminus\{0\}.
    \]
    We now show that this is impossible. Namely, if $f\in H^2$ satisfies
    \[
    \langle f,S_1^nf\rangle_{H^2}=0,
    \qquad n\in\mathbb Z\setminus\{0\},
    \]
    then \(f\equiv0\). 
    To see this, notice that, for every $n\in\mathbb Z\backslash\{0\}$,
    \begin{align*}
    0&=\langle f, S^n_1 f\rangle_{H^2}=\int_{-\pi}^{\pi}\chi_{[0,\pi]}(t)(\mathcal F f)(t)\overline{(\mathcal F (S^n_1 f)(t)}\operatorname{d}\!t= \int_{-\pi}^{\pi}\chi_{[0,\pi](t)}|\mathcal Ff(t)|^2 e^{-int}\operatorname{d}\!t.   
    \end{align*}
    Thus, all Fourier coefficients of the function $\chi_{[0,\pi]}|\mathcal F f|^2$
    vanish, except possibly the one corresponding to \(n=0\). Hence
    \(\chi_{[0,\pi]}|\mathcal F f|^2\) is constant almost everywhere on
    $[-\pi,\pi]$. Since it vanishes on \((-\pi,0)\), this constant must be zero.
    Therefore \(\mathcal F f=0\) almost everywhere on \([0,\pi]\), and hence
    \(f\equiv0\), as desired. This contradicts the existence of a non-zero vector \(f\in V\). We conclude that \(S_1\) has no non-trivial wandering subspace and, in particular, it is not a bilateral shift.
    \end{rmk}

    \begin{rmk}
    The decomposition of the boundary lattice into even and odd points also appears in the action of the Szeg\H{o} projection on the canonical basis of $\ell^2$.  Let $\text{e}_n$ be an element of  the canonical orthonormal basis of $\ell^2$, that is, $\text{e}_n(k)=\delta_n(k).$ We get
    \begin{equation}\label{E:Pe_n=K_n}
        P\text{e}_n(m)=\sum_{j=-\infty}^{+\infty}\text{e}_n(m-j)C(j)=C(m-n)=K_{n}(m), \qquad m\in\mathbb{Z}.
    \end{equation}
    In other words, $P\text{e}_n$ coincides with the kernel $K_n$ on $\mathbb{Z}$, hence the equality $P\text{e}_n=K_n$. Notice that
    \begin{equation}\label{Eq:ONBkernels}
        K_{-2k}(m)=\frac{1}{2\pi}\int_0^\pi e^{(m+2k)it}\operatorname{d}\!t=\frac{1}{\sqrt{2}}E_k(m), \qquad m\in\mathbb{Z}.
    \end{equation}
    In particular, $\{\sqrt 2 P e_{2k}\}_{k\in\mathbb{Z}}$ is orthonormal in $H^2$, as it coincides with the orthonormal basis that we have already discussed. The same holds considering the odd frequencies $\{\sqrt2 P e_{2k+1}\}_{k\in\mathbb{Z}}$, and this orthonormal basis would correspond to the set $\{\pi^{-1/2}e^{(2k+1)i(\cdot)}\}_{k\in\mathbb{Z}}$ in $L^2([0,\pi])$. In other words, after normalization, the Szeg\H{o} projection sends both the even
    and the odd parts of the canonical basis of $\ell^2$ onto orthonormal bases of
    the Hardy boundary space $J(H^2)$.
    \end{rmk} 
     
    \section{Discrete approximation of classical $H^2$-functions}\label{s:approximation}
    
    Let us denote by $\mathcal U$ the upper half-plane $\mathcal U=\{x+iy\in\mathbb C: y> 0\}$ and let $H^2(\mathcal U)$ be the classical Hardy space on the upper half-plane. The purpose of this section is to relate the discrete Hardy spaces with the classical space $H^2(\mathcal U)$. 
    By the classical Paley--Wiener characterization of $H^2(\mathcal U)$, \cite{PW}, we know that, given any $f\in H^2(\mathcal U)$, then
    \begin{equation}\label{Eq:PWhalfplane}
    f(x+iy)=\frac{1}{\sqrt{2\pi}} \int_0^\infty g(t) e^{-yt} e^{ixt}\operatorname{d}\! t, \qquad x+iy\in\mathcal{U},
    \end{equation}
    for some $g$ in $L^2(\mathbb R)$ with $\supp g\subseteq [0,+\infty)$. For the rest of this section, we fix such functions~$f,g$.
    
    The idea to approximate $f$ with functions on discrete lattices is the following. For $\varepsilon>0$, we consider the dilated lattice $\varepsilon\mathbb Z^2_+$. Notice that the smaller the parameter $\varepsilon$, the closer the points in $\varepsilon\mathbb Z^2_+$, and the fuller the grid $\varepsilon\mathbb Z^2_+$ becomes. Heuristically, first we approximate the function~$f$ on the points on the lattice $\varepsilon\mathbb Z^2_+$ for fixed $\varepsilon.$ Then, taking a limit as $\varepsilon\to 0^+$, we establish a uniform convergence result on compact sets of $\mathcal{U}$. 
    
    For $m,n\in\mathbb{Z}$, equation \eqref{Eq:PWhalfplane} gives
    \[
    f(\varepsilon m+i\varepsilon n)=\frac{1}{\sqrt{2\pi}} \int_0^\infty g(t) e^{-\varepsilon nt} e^{i\varepsilon mt}\operatorname{d}\! t.
    \]
    This motivates us to introduce the function defined on the lattice $\varepsilon\mathbb{Z}^2_+$
    \begin{equation}\label{eq:hvarepsilon}
    h(\varepsilon m+i\varepsilon n)=\frac{1}{\sqrt{2\pi}}\int_0^{\frac\pi\varepsilon}g(t)e_{\varepsilon t}(m+in)\operatorname{d}\!t= \frac{1}{\sqrt{2\pi}}\int_0^{\frac\pi\varepsilon} g(t) e^{im\varepsilon t}\Big(\frac{\cos(\varepsilon t)}{1+\sin(\varepsilon t)}\Big)^{n} \operatorname{d}\!t.
    \end{equation}
    
    We briefly develop the theory for discrete Hardy spaces on dilated lattices. For $\varepsilon>0$, we define
    \[
    H^2(\varepsilon\mathbb{Z}^2_+)=\{h\colon\varepsilon\mathbb{Z}^2_+\to\mathbb{C}\,\,\text{s.t.}\,\,\|h\|_{H^2(\varepsilon\mathbb{Z}^2_+)}^2=\sup_{n\geqslant 0}\sum_{m=-\infty}^{+\infty}\varepsilon\cdot|h(\varepsilon(m+in))|^2<\infty\}.
    \]
    One can check that
    \[
    H^2(\varepsilon\mathbb{Z}^2_+)=\{h\colon\varepsilon\mathbb{Z}^2_+\to\mathbb{C}\,|\,\exists f\in H^2(\mathbb{Z}^2_+)\,\,\text{s.t.}\,\,h(\varepsilon\cdot)=f\},\qquad \|h\|_{H^2(\varepsilon\mathbb{Z}^2_+)}=\sqrt{\varepsilon}\|f\|_{H^2(\mathbb{Z}^2_+)}.
    \]
    The factor $\varepsilon$ appears in the relation between the norms to take care of the increasing quantity of points in the grids $\varepsilon\mathbb{Z}^2$. In particular, introducing the $H^2(\mathbb{Z}^2_+)$-function
    \[
    \widetilde{h_\varepsilon}(m+in)=\frac{1}{\varepsilon\sqrt{2\pi}}\int_0^\pi g( t/\varepsilon)e^{imt}\Big(\frac{\cos t}{1+\sin t}\Big)^n \operatorname{d}\!t, \qquad m+in\in\mathbb{Z}^2_+,
    \]
    by a simple change of variable in \eqref{eq:hvarepsilon} we have that
    \begin{align*}
    h(\varepsilon m+i\varepsilon n)=
    \frac{1}{\varepsilon\sqrt{2\pi}}\int_0^\pi g( t/\varepsilon)e^{imt}\Big(\frac{\cos t}{1+\sin t}\Big)^n \operatorname{d}\!t=\widetilde{h_\varepsilon}(m+in).  
    \end{align*} 
    Then, by Theorem \ref{t:H^2_PW}, $h\in H^2(\varepsilon\mathbb{Z}^2_+)$ and
    \[
    \|h\|_{H^2(\varepsilon\mathbb{Z}^2_+)}^2=\varepsilon\|\widetilde{h_\varepsilon}\|_{H^2(\mathbb{Z}^2_+)}^2=\frac{1}{\varepsilon}\int_0^\pi |g(t/\varepsilon)|^2 \operatorname{d}\!t=\int_0^{\frac{\pi}{\varepsilon}} |g(s)|^2 \operatorname{d}\!s.
    \]
    In particular, by the dominated convergence theorem, the norm $\|h\|_{H^2(\varepsilon\mathbb{Z}^2_+)}$ tends to $\|f\|_{H^2(\mathcal{U})}=\|g\|_{L^2((0,\infty))}$, as $\varepsilon\to 0.$

    \begin{thm}\label{T:approximationdiscrete}
    Let $f$ be as in \eqref{Eq:PWhalfplane}, $0<\varepsilon\leqslant1$ and $h$ as in \eqref{eq:hvarepsilon}. Then, for every compact set $K\subseteq \mathcal U$ we have that
    \[
    \sup_{K\cap\varepsilon\mathbb{Z}^2_+}|f-h|
    \leqslant C_K \|g\|_{L^2((\frac{\pi}{2\varepsilon},\frac{\pi}{\varepsilon}))}+O(\varepsilon^{\frac{1}{4}}),
    \]
    where $C_K$ is a constant that only depends only on the compact $K$, not on $\varepsilon$ or $f$.
    \end{thm}
    Before the proof, we remark that since $g\in L^2$, the quantity
    \[
    \|g\|_{L^2((\frac{\pi}{2\varepsilon},\frac{\pi}{\varepsilon}))}^2=\int_0^\infty |g(t)|^2 \chi_{(\frac{\pi}{2\varepsilon},\frac{\pi}{\varepsilon})}(t)\operatorname{d}\!t
    \]
    tends to $0$ as $\varepsilon\to 0^+$. Hence, we could have written the statement of the theorem saying simply that $\sup_{K\cap\varepsilon\mathbb{Z}^2_+}|f-h|$ is $o(1)$ when $\varepsilon\to 0$. 
    \begin{proof}[Proof of Theorem \ref{T:approximationdiscrete}]
    There exist $0<a<b<\infty$ such that
    \(
    a\leqslant\operatorname{Im}z\leqslant b
    \) for every $z$ in $K$. Hence, whenever $\varepsilon m+i\varepsilon n\in K$, we  have 
    \begin{equation}\label{eq:bounds_n}
    0<a\leqslant\varepsilon n\leqslant b.
    \end{equation}
    We denote
    \[
    \varphi(t)=\frac{\cos(t)}{1+\sin(t)}, \qquad t\in[0,\pi].
    \]
    For $m,n$ such that $\varepsilon m+i\varepsilon n\in K$ we have that
    \begin{align*}
    \sqrt{2\pi}\big(f(\varepsilon m+i\varepsilon n)-h(\varepsilon m+i\varepsilon n)\big)
    &= \int_0^{\frac{\pi}{2\varepsilon}}
    g(t)
    (
    e^{-\varepsilon nt}-\varphi(\varepsilon t)^n
    )e^{i\varepsilon mt}\operatorname{d}\!t\\
    &\quad +  \int_{\frac{\pi}{2\varepsilon}}^{+\infty} g(t) e^{-\varepsilon n t }e^{i\varepsilon m t}\operatorname{d}\!t-
    \int_{\frac{\pi}{2\varepsilon}}^{\frac\pi\varepsilon} g(t)\varphi(\varepsilon t)^ne^{i\varepsilon m t}\operatorname{d}\!t\\
    &=I_{\varepsilon,m,n}
    +
    I\!I_{\varepsilon,m,n} 
    +
    I\!I\!I_{\varepsilon,m,n} 
    .
    \end{align*}
    We begin with $I\!I_{\varepsilon,m,n}$ and $
    I\!I\!I_{\varepsilon,m,n}$. Keeping in mind \eqref{eq:bounds_n}, we have
    \begin{align}\label{eq:II_bounds}
    \nonumber\sup_{\substack{m,n:\\
    \varepsilon m+i\varepsilon n\in K}} |I\!I_{\varepsilon,m,n}|&
    \leqslant \|g\|_{L^2((0,\infty))}\Big(\int_{\frac \pi{\varepsilon}}^{\infty}e^{-2\varepsilon n t}\operatorname{d}\!t\Big)^\frac{1}{2}\\
    &\leqslant \|g\|_{L^2((0,+\infty))}\Big(\int_{\frac \pi{\varepsilon}}^{+\infty}e^{-2a t}\operatorname{d}\!t\Big)^\frac{1}{2}=O(e^{-2a\frac{\pi}{\varepsilon}}),
    \end{align}
    for $0<\varepsilon<1$. About $
    I\!I\!I_{\varepsilon,m,n}$, by Cauchy--Schwarz,
    \begin{align}\label{Eq:IIIpt1}
        |I\!I\!I_{\varepsilon,m,n}|^2&\leqslant
        \|g\|_{L^2((\frac{\pi}{2\varepsilon},\frac{\pi}{\varepsilon}))}^2
        \int_{\frac{\pi}{2\varepsilon}}^{\frac{\pi}{\varepsilon}}
    |\varphi(\varepsilon t)|^{2n}
    \operatorname{d}\!t=
    \|g\|_{L^2((\frac{\pi}{2\varepsilon},\frac{\pi}{\varepsilon}))}^2\left(
      \frac1\varepsilon  \int_{\frac{\pi}{2}}^{\pi}
    |\varphi(t)|^{2n}
    \operatorname{d}\!t\right).
    \end{align}
    With the change of variable $s=\pi-t$, since $\phi$ enjoys the symmetry property in \eqref{Eq:symmetryphi}, we have that
    \begin{align}\label{Eq:IIIpt2}
      \frac1\varepsilon  \int_{\frac{\pi}{2}}^{\pi}
    |\varphi(t)|^{2n}
    \operatorname{d}\!t&= \frac1\varepsilon  \int_0^{\frac{\pi}{2}}
    |\varphi(s)|^{2n}
    \operatorname{d}\!s\leqslant\frac1\varepsilon  \int_0^{\infty}
    e^{-2ns}
    \operatorname{d}\!s = \frac{1}{2n\varepsilon}\leqslant\frac{1}{2a}.
    \end{align}
    We used the inequality $0\leqslant\phi(s)\leqslant e^{-s}$ for $0\leqslant s\leqslant\frac{\pi}{2}$. One can check it by differentiating.
    
    We turn to $I_{\varepsilon,m,n}$.  By Cauchy--Schwarz,
    \[
    \left|\int_0^{\frac{\pi}{2\varepsilon}}
    g(t)
    (
    e^{-\varepsilon nt}-\varphi(\varepsilon t)^n
    )e^{i\varepsilon mt}\operatorname{d}\!t\right|
    \leqslant
    \|g\|_{L^2((0,\infty))}
    \Big(
    \int_0^{\frac{\pi}{2\varepsilon}}
    |e^{-\varepsilon nt}
    -\varphi(\varepsilon t)^n
    |^2
    \operatorname{d}\!t
    \Big)^{1/2}.
    \]
    Since $\varepsilon<1,$ we split the integral in the right-hand side in two parts:
    \begin{align*}
        \int_0^{\frac{\pi}{2\varepsilon}}
    |e^{-\varepsilon nt}
    -\varphi(\varepsilon t)^n
    |^2
    \operatorname{d}\!t=&
    \int_0^{\frac{\pi}{2\sqrt\varepsilon}} |e^{-\varepsilon nt}-\varphi(\varepsilon t)^n|^2\operatorname{d}\!t +\int_{\frac{\pi}{2\sqrt\varepsilon}}^{\frac{\pi}{2\varepsilon}}
    |e^{-\varepsilon nt}
    -\varphi(\varepsilon t)^n
    |^2
    \operatorname{d}\!t.
    \end{align*}
    By a change of variable, we rewrite the first summand as
    \begin{align*}
    \int_0^{\frac{\pi}{2\sqrt\varepsilon}} |e^{-\varepsilon nt}-\varphi(\varepsilon t)^n|^2\operatorname{d}\!t&=\frac1\varepsilon\int_0^{\frac{\sqrt\varepsilon\pi}{2}} |e^{- nt}-\varphi(t)^n|^2\operatorname{d}\!t.
    \end{align*}
    Checking the Taylor expansions, one can show that $e^{-t}-\varphi(t)=t^3+o(t^3)$, as $t\to 0^+$. In particular, 
    \[
    |e^{-t}-\varphi(t)|\leqslant Ct^3, \qquad t\in(0,\sqrt{\varepsilon}\pi),
    \]
    for some constant $C>0$ that does not depend on $t$ nor $\varepsilon$. 
    Hence, from the identity
    \[
    A^n-B^n=(A-B)\sum_{k=0}^{n-1} A^{n-k-1}B^k,
    \]
    and the fact that $|e^{-t}|,|\varphi(t)|\leqslant 1$ for $t\in [0,\pi]$, we conclude that
    \begin{equation}\label{eq:I_main}
    \int_0^{\frac{\pi}{2\sqrt\varepsilon}} |e^{-\varepsilon nt}-\varphi(\varepsilon t)^n|^2\operatorname{d}\!t\leqslant \frac{n^2}{\varepsilon}\int_0^{\frac{\sqrt\varepsilon\pi}{2}}|e^{-t}-\varphi(t)|^2\operatorname{ d}\! t\leqslant C^2\frac{b^2}{\varepsilon^3}\int_0^{\frac{\sqrt\varepsilon\pi}{2}}t^6\operatorname{d}\!t =O(\sqrt{\varepsilon}).
    \end{equation}
    In the last inequality we also used the fact that $\varepsilon n\leqslant b$. We move on to the last summand.
    
    We use again that 
    \(
    e^{-t}\geqslant \varphi(t)\geqslant 0
    \)
    for $t$ in $[0,\frac\pi2]$.  
    In particular, whenever $t\in[0,\frac{\pi}{2\varepsilon}]$, 
    we have the estimate 
    \[
    0\leqslant \varphi(\varepsilon t)^n
    \leqslant e^{-\varepsilon nt}
    \leqslant e^{-at}.
    \]
    Thus,
    \[
    |
    e^{-\varepsilon nt}
    -
    \varphi(\varepsilon t)^n
    |^2
    \leqslant
    (e^{-\varepsilon nt}
    +
    \varphi(\varepsilon t)^n)^2
    \leqslant
    4e^{-2at},\qquad t\in[0,\frac{\pi}{2\varepsilon}],
    \]
    and we conclude that
    \begin{equation}\label{eq:I_tail}
    \sup_{a\leqslant \varepsilon n\leqslant b}\int_{\frac{\pi}{2\sqrt\varepsilon}}^{\frac{\pi}{2\varepsilon}}
    |e^{-\varepsilon nt}
    -\varphi(\varepsilon t)^n
    |^2
    \operatorname{d}\!t
    \leqslant
    4\int_{\frac{\pi}{2\sqrt\varepsilon}}^{+\infty}e^{-2at}\operatorname{d}\!t =O(e^{-a\frac{\pi}{\sqrt\varepsilon}}).
    \end{equation}
    
    From \eqref{eq:II_bounds}, \eqref{Eq:IIIpt1}, \eqref{Eq:IIIpt2}, \eqref{eq:I_main} and \eqref{eq:I_tail} we conclude that
    \[
    \sup_{\substack{m,n:\\
    \varepsilon m+i\varepsilon n\in K}}|f(\varepsilon m+i\varepsilon n)-h(\varepsilon m+i\varepsilon n)|\leqslant \frac{1}{\sqrt{2a}}\|g\|_{L^2((\frac{\pi}{2\varepsilon},\frac{\pi}{\varepsilon}))}+O(\varepsilon^{\frac{1}{4}}),
    \]
    hence the proof.
    \end{proof}
    
    Given a real number $x$ we write $W(x):=\lfloor x+\frac 12\rfloor$, where $\lfloor\cdot\rfloor$ is the usual floor function.
    
    \begin{cor}
    Let $f$ be as in \eqref{Eq:PWhalfplane}, $h$ as in \eqref{eq:hvarepsilon} and $K\subseteq\mathcal U$ compact. For $z=x+iy$ in $K$, set
    \[
    m_\varepsilon(z)=W\!\left(\frac{x}{\varepsilon}\right),
    \qquad
    n_\varepsilon(z)=W\!\left(\frac{y}{\varepsilon}\right),\qquad z_\varepsilon
    =
    \varepsilon m_\varepsilon(z)
    +
    i\varepsilon n_\varepsilon(z).
    \]
    Then
    \[
    \lim_{\varepsilon\to 0^+}\sup_{z\in K}
    \left|
    f(z)-h(z_\varepsilon)
    \right|=0.
    \]
    \end{cor}
    \begin{proof}
    Since
    $
    \left|x-\varepsilon m_\varepsilon(z)\right|
    \leqslant \varepsilon/2$ and $
    \left|y-\varepsilon n_\varepsilon(z)\right|
    \leqslant \varepsilon/2$, we have $|z-z_\varepsilon|
    \leqslant \varepsilon$. Hence,
    \[
    \lim_{\varepsilon\to 0^+}\sup_{z\in K}|z-z_\varepsilon|=0.
    \]
    Since $K\cap \partial\mathcal U=\emptyset$, there exists a compact set
    $K'\subseteq\mathcal U$, with $K'\cap \partial\mathcal U=\emptyset$, such that
    $z_\varepsilon\in K'$ for every $z\in K$ and every sufficiently small
    $\varepsilon>0$. By the triangle inequality,
    \[
    |f(z)-h(z_\varepsilon)|
    \leqslant
    |f(z)-f(z_\varepsilon)|
    +
    |f(z_\varepsilon)-h(z_\varepsilon)|.
    \]
    Therefore,
    \[
    \sup_{z\in K}|f(z)-h(z_\varepsilon)|
    \leqslant
    \sup_{z\in K}|f(z)-f(z_\varepsilon)|+
    \sup_{\substack{m,n:\\
    \varepsilon m+i\varepsilon n\in K'}}
    |f(\varepsilon m+i\varepsilon n)
    -
    h(\varepsilon m+i\varepsilon n)|.
    \]
    The first term tends to $0$ because $f$ is uniformly continuous on $K'$. The second term tends to $0$ by the previous theorem applied to $K'$. Hence,
    \[
    \lim_{\varepsilon\to 0^+}\sup_{z\in K}|f(z)-h(z_\varepsilon)|=0.
    \]
    \end{proof}
    
    \section{Characterization of $J(H^p)$ and duality}\label{s:Hp_duality}
    We extend the Paley--Wiener characterization of Theorem~\ref{t:H^2_PW}
    to the spaces $H^p$, $1\leqslant p<\infty$. We shall use the basic facts on
    periodic distributions and their Fourier transform recalled in
    Subsection~\ref{ss:distributions}. After that, we briefly discuss the boundedness of the
    Szeg\H{o} projection on $\ell^p$ and use it to identify the dual space of
    $H^p$ for $1<p<\infty$.
    
    \begin{thm}\label{t:H^p_PW}
    For every $1\leqslant p<\infty$,
    \[
    J(H^p)=\{f\in \ell^p(\mathbb Z): \operatorname{supp}\mathcal F f\subseteq [0,\pi]\},
    \]
    where $\mathcal{F}$ is the Fourier transform in the sense of distributions.
    \end{thm}
    
    \begin{proof}
    The inclusion
    \[
    J(H^p)\subseteq \{f\in \ell^p(\mathbb Z): \operatorname{supp}\mathcal F f\subseteq [0,\pi]\}
    \]
    is in essence analogous to that in Theorem \ref{t:H^2_PW}, with the only foresight that the Fourier transform is intended in the distributional sense and thus one has to proceed with more caution. 
    
    By \eqref{eq:cauchy_half_plane} we have the identity $f=f\ast C$, for every $f\in J(H^p)$. 
    To prove that $\operatorname{supp}\mathcal F f\subseteq [0,\pi]$ we take $\psi\in C_c^\infty(-\pi,0)$, and we show that $\mathcal{F}f(\psi)=\mathcal{F}(f\ast C)(\psi)=0$. By \eqref{E:distributionalFourier},
    \begin{align}\label{E:distribuzionenulla0}
      \nonumber  \langle \mathcal F(f\ast  C),\psi\rangle &=\sum_{n=-\infty}^{+\infty} \left(\sum_{m=-\infty}^{+\infty} f(m)C(n-m)\right) \mathcal{F}^{-1}\psi(n)\\
        &=\sum_{m=-\infty}^{+\infty} f(m)\left(\sum_{n=-\infty}^{+\infty} C(n-m) \mathcal{F}^{-1}\psi(n)\right),
    \end{align}
    where in the last equality we can use Fubini's Theorem to exchange the order of the sums, since $f\in\ell^p,$ $C\in\bigcap_{q>1}\ell^q$ (see Lemma \ref{l:C_expansion}) and $\mathcal{F}^{-1}\psi$ is rapidly decaying. Since $\mathcal{F}C=(2\pi)^{-\frac{1}{2}}\chi_{[0,\pi]}$ (see \eqref{E:C=F^{-1}}) and $\mathcal{F}$ is unitary, we have that for every $m\in\mathbb{Z}$
    \begin{align*}
        \sum_{n=-\infty}^{+\infty} C(n-m) \mathcal{F}^{-1}\psi(n) &= \int_{-\pi}^{\pi}\mathcal{F}\big(C(\cdot-m)\big)(t)\psi(t)\operatorname{d}\!t\\ 
        &=\frac{1}{\sqrt{2\pi}}\int_{0}^{\pi}e^{imt}\psi(t)\operatorname{d}\!t=0,
    \end{align*}
    thus proving that $\operatorname{supp}\mathcal F f\subseteq [0,\pi]$.

    To prove the converse we need a little extra work. Let
    \[
    f\in \ell^p(\mathbb Z),
    \qquad
    \operatorname{supp}\mathcal F f\subseteq[0,\pi],
    \]
    and set
    \[
    \widetilde f(m+in)=(f*P_n)(m)=\sum_{j=-\infty}^{+\infty}f(j)P_n(m-j),
    \qquad n\geqslant0,
    \]
    where $P_n$ is the Poisson kernel.
    Since \(\|P_n\|_{\ell^1}=1\) for every $n$, Young's
    inequality gives
    \[
    \sup_{n\geqslant 0}\|\widetilde f(\cdot+in)\|_{\ell^p}
    =\sup_{n\geqslant 0}
    \|f*P_n\|_{\ell^p}
    \leqslant
    \|f\|_{\ell^p}.
    \]
    Moreover, since \(P_0=\delta_0\), we have \(J\widetilde f=f\). Hence, to prove that $\widetilde f\in H^p$ it only remains to prove that $\widetilde f$ is discrete holomorphic. Namely, we have to  show that, for every $m+in\in\mathbb Z^2_+$,
    \[
    \widetilde f(m+1+i(n+1))-\widetilde f(m+in)+i\big(\widetilde f(m+i(n+1))-\widetilde f(m+1+in)\big)=0.
    \]
    Equivalently,
    \[
    f*
    \big(S_1 P_{n+1}-P_n+iP_{n+1}-iS_1 P_n\big)(m)=0,\qquad m+in\in\mathbb Z^2_+,
    \]
    where $S_1 P_n(m)=P_n(m+1)$. 
    Thus, setting
    \[
    g_n
    =
    S_1 P_{n+1}-P_n+iP_{n+1}-iS_1 P_n,\qquad n\geqslant 0,
    \]
    it is enough to prove 
    $\mathcal F(f\ast g_n)=0$ for every $n\geqslant 0$, where the Fourier transform is again to be taken in the sense of distributions. First, we prove that the support $\supp(\mathcal F(f\ast  g_n))$ is finite. 
    
    Since for every $n\geqslant 0$ the Poisson kernel $P_n$ is a well-defined $\ell^1$-function, its distributional Fourier transform coincides with its proper Fourier transform. Considering $\varphi_n(t)=\Big(\frac{\cos|t|}{1+\sin|t|}\Big)^n$ as in \eqref{eq:Poisson_coefficient}, we have that 
    \begin{align*}
     P_n(m)=\frac{1}{2\pi}\int_{-\pi}^{\pi}\varphi_n(t) e^{imt}\operatorname{d}\!t=\frac{1}{\sqrt{2\pi}}\mathcal{F}^{-1}\varphi_n(m).
    \end{align*}
    Thus,
    \begin{align*}
     \sqrt{2\pi} \mathcal F g_n(t)&= e^{it} \Big(\frac{\cos|t| }{1+\sin|t|}\Big)^{n+1}- \Big(\frac{\cos|t| }{1+\sin|t|}\Big)^{n}+i \Big(\frac{\cos|t| }{1+\sin|t|}\Big)^{n+1}-ie^{it} \Big(\frac{\cos|t| }{1+\sin|t|}\Big)^{n}\\
    &=\Big(\Big(\frac{\cos|t| }{1+\sin|t|}\Big)(e^{it}+i)-1-ie^{it}\Big)\Big(\frac{\cos|t| }{1+\sin|t|}\Big)^n.
    \end{align*}
    This identity and the identity
    \[
    \frac{\cos t}{1+\sin t}=\frac{1+ie^{it}}{i+e^{it}},\quad t\in [0,\pi],
    \]
    guarantee that $\mathcal F g_n(t)=0$ for $t\in [0,\pi]$. Now take $\eta$ in $C_c^\infty((0,\pi)).$ As it was done in \eqref{E:distribuzionenulla0}, since $g_n\in\ell^1$ for every $n$ we can write
    \begin{equation}\label{E:distribuzionenulla1}
         \langle \mathcal F(f\ast  g_n),\eta\rangle =\sum_{m=-\infty}^{+\infty} \left(f(m)\sum_{k=-\infty}^{+\infty} g_n(k-m) \mathcal{F}^{-1}\eta(k)\right).
    \end{equation}
    
    We rewrite the right-hand side in \eqref{E:distribuzionenulla1}. For $n\geqslant 0$,
    \begin{align*}
        \sum_{k=-\infty}^{+\infty} g_n(k-m) \mathcal{F}^{-1}\eta(k)&=\frac{1}{\sqrt{2\pi}}\sum_{k=-\infty}^{+\infty} g_n(k) \int_{-\pi}^{\pi}\eta(x)e^{ i(k+m)x}\operatorname{d}\!x\\
        &=\frac{1}{\sqrt{2\pi}}\int_{-\pi}^{\pi}\eta(x)\left(\sum_{k=-\infty}^{+\infty} g_n(k)e^{ikx}\right)e^{imx}\operatorname{d}\!x \\
        &=\int_{-\pi}^{\pi}\eta(x)\mathcal{F}g_n(x)e^{imx}\operatorname{d}\!x=0,
    \end{align*}
    since $\mathcal F g_n(t)=0$ for $t\in [0,\pi]$ and $\eta\in C_c^\infty((0,\pi)).$ We remark that in the second equality we can exchange the series and the integral because $g_n\in\ell^1$, and we obtain the third equality since $\mathcal{F}g_n$ is continuous and piecewise $C^1$. We proved that $\supp(\mathcal F(f\ast  g_n))\subseteq [-\pi,0]\cup\{\pi\}$.

    Now, take $\psi$ in $C_c^\infty((-\pi,0)).$ Since $g_n\in\ell^1$ for every $n$ we can write
    \[
     \langle \mathcal F(f\ast  g_n),\psi\rangle =\sum_{m=-\infty}^{+\infty} g_n(m)\left(\sum_{k=-\infty}^{+\infty} f(k-m) \mathcal{F}^{-1}\psi(k)\right).
    \]
    We focus on the right-hand side. For every $m$,
    \begin{align*}
        \sum_{k=-\infty}^{+\infty} f(k-m) \mathcal{F}^{-1}\psi(k)&=\sum_{k=-\infty}^{+\infty} f(k) \mathcal{F}^{-1}\psi(k+m)\\
        &=\sum_{k=-\infty}^{+\infty} f(k) \mathcal{F}^{-1}\big(e^{im(\cdot)}\psi\big)(k)=\langle \mathcal{F}f,e^{im(\cdot)}\psi\rangle=0,
    \end{align*}
    since by assumption $\supp(\mathcal Ff)\subset [0,\pi]$ and $\psi\in C_c^\infty((-\pi,0))$.

     We conclude that $\supp(\mathcal F(f\ast g_n))$ is finite, and it is contained in $\{-\pi,0,\pi\}.$ Since every distribution on the torus has finite order (see \cite[Theorem 6.2]{beals}), and every distribution of finite support and finite order is a finite sum of derivatives of Dirac deltas (see \cite[Theorem 2.3.4]{hormander1983analysis}), we conclude that
    \[
    \mathcal F(f\ast g_n) = \sum_{k=0}^{N_0} \alpha_k \delta_0^{(k)} + \sum_{k=0}^{N_1} \beta_k \delta_{\pi}^{(k)},
    \]
    where 
    \[
    \delta_y^{(k)}(\psi):=(-1)^k \psi^{(k)}(y), \qquad \psi\in C^\infty(\mathbb{T}).
    \]
    We can disregard the point $-\pi$ in the support because of the periodicity. Computing the (distributional) Fourier coefficients of a Dirac delta yields
    \[
    \mathcal{F}^{-1}(\delta_y^{(k)})(h)=\frac{1}{\sqrt{2\pi}}(\delta_y^{(k)})(e^{-ih(\cdot)})=\frac{1}{\sqrt{2\pi}} (-1)^k(-ih)^k e^{-ihy}, \qquad h\in\mathbb Z.
    \]
    In particular, the growth in $h$ of the Fourier coefficients of $\mathcal F(f\ast g_n)$ is that of a polynomial. On the other hand, the $h$-th Fourier coefficient of $\mathcal F(f\ast g_n)$ is exactly $f\ast g_n(h)$, and this is an $\ell^p$-sequence. Then, necessarily, $\mathcal{F}(f\ast g_n)=0$.
    \end{proof}
    
    We discuss the Szeg\H o projection $P\colon\ell^p\to J(H^p)$, introduced in the case $p=2$ in Theorem~\ref{t:H^2_PW}. 
    
    \begin{prop}\label{P:szegoproj}
        Let $1<p<\infty$. The Szeg\H o projection 
        \begin{equation*} 
        Pf=f\ast C=\mathcal{F}^{-1}\big(\chi_{[0,\pi]}\mathcal{F}f\big), \qquad f\in\ell^p\cap\ell^2,
    \end{equation*}
    extends to a bounded operator $P\colon\ell^p\to\ell^p$. Moreover, the range $P(\ell^p)$ coincides with $J(H^p)$ and $Pf=f$ for every $f$ in $J(H^p)$.  
    \end{prop}
    \begin{proof}
    By \eqref{Eq:P_Hilbert}, it is clear that the boundedness of $P\colon\ell^p\to\ell^p$ is equivalent to that of $\mathcal{H}_{\operatorname{odd}}$. Since $\mathcal H_{\operatorname{odd}}$ is a discrete Hilbert--type transform, its $\ell^p$-boundedness is well-known and classical; see for instance \cite{HMW} or \cite{BK}. Consequently, the Szeg\H o projection $P$ extends to a bounded operator on $\ell^p$ for $1<p<\infty$. That $Pf=f$ for every $f\in J(H^p)$ follows from \eqref{E:reproducingformulaconvolution}. In particular, $P(\ell^p)\supseteq J(H^p)$. Notice also that the formula 
    \begin{equation}\label{eq:hilbert_lp}
    f(m)=Pf(m)=\frac{2i}{\pi}\sum_{j\in2\mathbb Z+1}\frac{f(m-j)}{j}, \qquad m\in\mathbb{Z},
    \end{equation}
    holds for $\ell^p$ functions as well. 
    
    To show the reverse inclusion, recall that $P\text{e}_n=K_n$, see \eqref{E:Pe_n=K_n}. Then, since $\{\text{e}_n\}_{n\in\mathbb{Z}}$ is a Schauder basis for $\ell^p$, we have that
    \[
    P(\ell^p)\subseteq \overline{\operatorname{span}\{P\text{e}_n\}_{n\in\mathbb{Z}}}=\overline{\operatorname{span}\{K_n\}_{n\in\mathbb{Z}}}\subseteq \overline{J(H^p)}=J(H^p),
    \]
    since $K_n\in J(H^p)$ for every $n\in\mathbb{Z}$, see Remark \ref{Rmk:kernelsHp}.
    \end{proof}
    
    After settling these standard properties for the Szeg\H o projection on $\ell^p$, we are in a position to characterize the duality of $H^p$ spaces. The proof is the adaptation of a classical argument, we include for the sake of completeness.
    
    \begin{thm}
        Let $1<p<\infty$. The dual space $(H^p)^*$ can be identified with $H^{p'},$ where $p'$ is the conjugate exponent of $p$. More precisely, there exists an isomorphism of Banach spaces $\Phi\colon H^{p'}\to(H^p)^*$.
    \end{thm}
    
    \begin{proof}
    For simplicity, set $q=p'$. Given $f\in H^p$ and $g\in H^q,$ consider the duality given by 
    \[
    \langle f,g\rangle = \sum_{m=-\infty}^{+\infty}f(m)\overline{g(m)}.
    \]
    For $g\in H^q$, the operator
    \[
    \Phi_g \colon H^p\to\mathbb{C}, \qquad f\mapsto \langle f,g\rangle,
    \]
    is well-defined and bounded, with $\|\Phi_g\|_{(H^p)^*}\leqslant \|g\|_{H^q}$, as a consequence of the H\"older inequality. In particular, it is well-defined and bounded the mapping
    \[
    \Phi \colon H^q \to (H^p)^*, \qquad g\mapsto \Phi_g.
    \]
    To show that it is surjective we use a standard argument using the Szeg\H o projection $P$ (see Proposition \ref{P:szegoproj} for its properties) and the well-known results on the duality of $\ell^p$ spaces. To distinguish between the duality of $\ell^p$ and $H^p$ spaces, we introduce the two different notations $\langle\cdot,\cdot\rangle_{\ell^2}$ and $\langle\cdot,\cdot\rangle_{H^2}$, even though they represent essentially the same quantity. 
    
    Consider $\Lambda$ in $(H^p)^*$. Since $H^p$ is isometrically isomorphic to $J(H^p)\subseteq \ell^p$, we can transfer the operator $\Lambda$ to $J(H^p)$, that is a closed subspace of $\ell^p$. Using Hahn--Banach, we can extend $\Lambda$ to an operator $\widetilde{\Lambda}$ on $\ell^p,$ preserving its norm. Now, using the duality of $\ell^p$, we know that there exists $\phi\in\ell^q$ such that
    \[
    \widetilde{\Lambda} f=\langle f,\phi\rangle_{\ell^2}=\sum_{m=-\infty}^{+\infty}f(m)\overline{\phi(m)}, \qquad f\in\ell^p,
    \]
    and $\|\phi\|_{q}=\|\widetilde{\Lambda}\|_{(\ell^p)^*}=\|\Lambda\|_{(H^p)^*}$. Since $P\colon \ell^2\to J(H^2)$ is an orthogonal projection, we have that
    \[
    \langle Pf,h\rangle_{\ell^2} = \langle f,Ph\rangle_{\ell^2} ,\qquad f,h\in\ell^2.
    \]
    In particular, approximating $f\in\ell^p$ and $h\in\ell^q$ with sequences of compact support $(f_n)_n$ and $(h_n)_n$, respectively, by the boundedness of $P$ on $\ell^p$ and $\ell^q$ we have that
    \[
    \langle Pf,h\rangle_{\ell^2} = \lim_n \langle Pf_n,h_n\rangle_{\ell^2} = \lim_n \langle f_n,Ph_n\rangle_{\ell^2} = \langle f,Ph\rangle_{\ell^2}. 
    \]
    Consider now $g:=J^{-1}P\phi\in H^q$. This is well-defined, for $P\phi\in J(H^q)$. Then, for $f\in H^p,$ since $PJf=Jf$ we have that
    \begin{align*}
        \Lambda(f)=\langle PJf, \phi\rangle_{\ell^2}=\langle Jf, P\phi\rangle_{\ell^2}=\langle f, J^{-1}P\phi\rangle_{H^2}=\langle f, g\rangle_{H^2}=\Phi_g(f).
    \end{align*}
    We proved that $\Lambda = \Phi_g$. We also have that
    \[
    \|\Phi_g\|_{(H^p)^*}\leqslant \|g\|_{H^q}\leqslant \|P\|_{\ell^q\to \ell^q}\|\phi\|_{\ell^q}=\|P\|_{\ell^q\to \ell^q}\|\Lambda\|_{(H^p)^*}=\|P\|_{\ell^q\to \ell^q}\|\Phi_g\|_{(H^p)^*},
    \]
    proving that $\Phi$ is an isomorphism of Banach spaces.
    \end{proof}
    
    \section{Uniqueness and sampling on horizontal lines}\label{s:sampling}
    
    Questions related to zero sets, uniqueness sets and sampling are natural in this discrete setting, and may reveal phenomena with no direct
    counterpart in the classical continuous theory. A systematic study of these
    questions is beyond the scope of the present paper. Here we record a few first
    consequences of the boundary structure developed above. In what follows, we
    often identify a function $f$ in $H^p$ with its boundary values and omit the
    restriction operator~$J$ from the notation.
    
    We first record a uniqueness consequence of the boundary relation obtained from
    the Szeg\H{o} projection. In Proposition~\ref{p:uniqueness} this was proved for
    $H^2$. Since formula~\eqref{eq:f_even_odd} remains valid in $H^p$, $1<p<\infty$,
    by Proposition~\ref{P:szegoproj}, the same argument gives the following
    extension.
    
    \begin{prop}
    For every fixed $n\geqslant 0$, the sets $\Lambda_{n}^{\operatorname{even}}$ and $\Lambda_{n}^{\operatorname{odd}}$ are uniqueness sets for $H^p,1<p<\infty$. Namely, given $f$ in $H^p$, if either  
    $
    f|_{\Lambda_n^{\operatorname{even}}}\equiv 0$ or $
    f|_{\Lambda_n^{\operatorname{odd}}}\equiv 0,
    $
    then $f\equiv 0$.
    \end{prop}

In the Hilbert space setting, a curious phenomenon occurs. Given $f$ in $H^2$, at every height $n\geqslant 0$, the series 
\[
\sum_{m=-\infty}^{+\infty}|f(m+in)|^2
\]
  splits exactly on the even and odd numbers. When $n=0$, this yields a sampling result.    
    \begin{prop}\label{p:sampling} Given $f$ in $H^2$ and $n\geqslant 0$, the following identities hold:
     \[
    \sum_{m=-\infty}^{+\infty} |f(2m+in)|^2=\sum_{m=-\infty}^{+\infty}|f(2m+1+in)|^2=\frac12\sum_{m=-\infty}^{+\infty}|f(m+in)|^2.
     \]
     In particular, at the boundary $n=0$,
     \[
     \sum_{m=-\infty}^{+\infty} |f(2m)|^2=\sum_{m=-\infty}^{+\infty}|f(2m+1)|^2=\frac12\|f\|_{H^2}^2,
     \]
     and the sets $\Lambda^{\operatorname{even}}_{0}$ and $\Lambda^{\operatorname{odd}}_{0}$ are sampling sets for $H^2$.  
    \end{prop}
    \begin{proof}
   We begin with $n=0$. We prove the identity for $\Lambda^{\operatorname{even}}_0$. The result for $\Lambda^{\operatorname{odd}}_0$ will then follow. The proof is essentially a consequence of the identity \eqref{Eq:ONBkernels}, that connects the orthonormal basis $\{E_k\}_{k\in\mathbb{Z}}$ with the reproducing kernels. We have that
    \[
    \|f\|_{H^2}^2=\sum_{m=-\infty}^{+\infty}|\langle f,E_m\rangle|^2=2\sum_{m=-\infty}^{+\infty}|\langle f,K_{-2m}\rangle|^2=2\sum_{m=-\infty}^{+\infty}|f(-2m)|^2.
    \]
    Now, if $n> 0$, set $f_n:=f(\cdot+in)$. Clearly, $f_n\in H^2$ and we can apply the previous result obtaining
    \[
    \sum_{m=-\infty}^{+\infty} |f(2m+in)|^2=\sum_{m=-\infty}^{+\infty} |f_n(2m)|^2=\frac12\|f_n\|_{H^2}^2=\frac12\sum_{m=-\infty}^{+\infty}|f(m+in)|^2.
    \]
    The identity for odd numbers at height $n$ follows.

    \end{proof}
    This exact splitting of the boundary norm between the two parity classes has 
    no direct analogue in the classical Hardy space on the upper half-plane. We actually have also a $H^p$-version, $1<p<\infty$, of the sampling result above.
    \begin{prop}\label{p:sampling_Hp}
    The sets $\Lambda^{\operatorname{even}}_0$ and
    $\Lambda^{\operatorname{odd}}_0$ are sampling sets for $H^p$,
    $1<p<\infty$. More precisely, there exists a constant $C_p>0$ such that,
    for every $f\in H^p$,
    \[
    \|f|_{\Lambda^{\operatorname{even}}_0}\|_{\ell^p}
    \leqslant
    \|f\|_{H^p}
    \leqslant
    C_p\|f|_{\Lambda^{\operatorname{even}}_0}\|_{\ell^p}.
    \]
    The same conclusion holds with $\Lambda^{\operatorname{even}}_0$ replaced by
    $\Lambda^{\operatorname{odd}}_0$.
    \end{prop}
    
    \begin{proof}
     Given $f$ in $H^p$, we set $f_e(m)=f(2m)$ and $f_o(m)=f(2m+1)$. Then, writing 
    \eqref{eq:hilbert_lp} more explicitly,\begin{align*}
    f_o(m)&=  f(2m+1)=\frac{2i}{\pi}\sum_{j=-\infty}^{+\infty}\frac{f(2m+1-(2j+1))}{2j+1}  \\
     &=\frac{2i}{\pi}\sum_{j=-\infty}^{+\infty} \frac{f(2m-2j)}{2j+1}=\frac{2i}{\pi}\sum_{j=-\infty}^{+\infty}\frac{f_e(m-j)}{2j+1}.
     \end{align*}
     Thus, $f_o$ is obtained by an operator applied to $f_e$. This operator is a Hilbert transform-type operator, thus, in particular, it is $\ell^p$-bounded for $1<p<\infty$. Hence, $\|f_o\|_{\ell^p}\leqslant C\| f_e\|_{\ell^p}$. Therefore, 
     \begin{align*}
    \|f_e\|_{\ell^p}^p&\leqslant \|f\|^p_{\ell^p}=\|f_e\|_{\ell^p}^p+\|f_o\|_{\ell^p}^p\leqslant (1+C^p)\|f_e\|_{\ell^p}^p.
    \end{align*}
     Similarly, we also obtain
    \[
    \|f_o\|_{\ell^p}^p\leqslant \|f\|^p_{\ell^p}\leqslant (1+C^p)\|f_o\|_{\ell^p}^p.
    \]
    \end{proof}

    We conclude the discussion on sampling in $H^p$ by showing that, at least for $H^2$, the sets  $\Lambda^{\operatorname{even}}_n$ and $\Lambda^{\operatorname{odd}}_n$ are not sampling sets for $n\geqslant1$, that is, the sampling result in Proposition \ref{p:sampling} is a genuine boundary phenomenon. Although the splitting of the $\ell^2$-norms occurs at every height $n\geqslant 0$, only the boundary sets control the $H^2$-norm.
    \begin{prop}\label{p:not_sampling_positive_height}
    The sets $\Lambda^{\operatorname{even}}_n$ and $\Lambda^{\operatorname{odd}}_n$ are not sampling sets for $H^2$ if $n\geqslant 1$. 
    \end{prop}
    \begin{proof}
     Again, we only focus on the sets $\Lambda^{\operatorname{even}}_n$. We prove that there does not exist a constant $c>0$ such that 
     \begin{equation}\label{eq:lower_sampling}
     c\|f\|^2_{H^2}\leqslant \sum_{m=-\infty}^{+\infty}|f(2m+in)|^2
     \end{equation}
     for every $f$ in $H^2$. Indeed, for every $0<\varepsilon<\pi/2$, let us consider the function $g_{\varepsilon}(t)=(\sqrt{ 2\varepsilon})^{-1} \chi_{\varepsilon}(t)$ where $\chi_\varepsilon$ denotes the characteristic function of the interval $[\pi/2-\varepsilon,\pi/2+\varepsilon]$. Then $\|g_\varepsilon\|_{L^2}=1$ and the associated $H^2$-function
     \[
     h_\varepsilon(m+in)=\frac{1}{\sqrt{2\pi}} \int_0^\pi g_\varepsilon(t) e^{imt}\Big(\frac{\cos t}{1+\sin t}\Big)^n\operatorname{d}\!t, \qquad m+in\in\mathbb{Z}^2_+,
     \]
     satisfies $\|h_\varepsilon\|_{H^2}=1$ as well. However, for $n\geqslant1$, by Proposition \ref{p:sampling} and the unitarity of $\mathcal{F}$,
     \begin{align*}
         \sum_{m=-\infty}^{+\infty}|h_\varepsilon(2m+in)|^2&=\frac12\sum_{m=-\infty}^{+\infty}|h_\varepsilon(m+in)|^2\\
         &=
         \frac{1}{4\varepsilon}\int_{\frac\pi2-\varepsilon}^{\frac\pi2+\varepsilon}\Big|\frac{\cos t}{1+\sin t}\Big|^{2n}\operatorname{d}\! t\to \frac{1}{2}\Big|\frac{\cos \frac\pi2}{1+\sin \frac\pi2}\Big|^{2n}=0,
     \end{align*}
     as $\varepsilon\to 0$.
     Thus, there cannot exist a constant $c>0$ such that \eqref{eq:lower_sampling} holds for every $h$ in $H^2$.
     \end{proof}

    \section{Weighted Bergman spaces}\label{s:Bergman}
    The Hardy spaces studied in the previous sections naturally lead to other Hilbert spaces of discrete holomorphic functions on the upper half-lattice. In this final section we briefly discuss Bergman-type spaces. Our goal is not to develop a complete theory, but rather to record some basic facts and to indicate two natural ways of introducing weights. The first one follows the classical idea of using powers of the diagonal Bergman kernel as weights. The second one is adapted to the Paley--Wiener representation and uses powers of the spectral weight appearing in the Bergman norm. We show that, for a suitable range of parameters, these two constructions give the same spaces, up to equivalence of norms.
    
    The unweighted Bergman space is naturally defined by
    \[
    A^2(\mathbb{Z}^2_+)
    =
    \Big\{
    f\in \operatorname{Hol}(\mathbb Z^2_+):
    \|f\|_{A^2}^2
    :=
    \sum_{n=0}^{+\infty}
    \sum_{m=-\infty}^{+\infty}
    |f(m+in)|^2
    <\infty
    \Big\}.
    \]
    Since
    \[
    \sup_{n\geqslant 0}
    \sum_{m=-\infty}^{+\infty}|f(m+in)|^2
    \leqslant
    \sum_{n=0}^{+\infty}
    \sum_{m=-\infty}^{+\infty}
    |f(m+in)|^2,
    \]
    we have the bounded inclusion
    $A^2\subseteq H^2$.
    Therefore, the Paley--Wiener theorem for $H^2$ can be directly used to obtain a
    spectral characterization of $A^2$. 
    Using the same notation of Theorem~\ref{t:H^2_PW}, we have the following.
    \begin{thm}\label{t:B^2_PW}
    The Fourier transform $\mathcal F$ maps $J(A^2)$ onto
    \[
    X=\left\{g\in L^2(\mathbb T): \supp g\subseteq [0,\pi], \int_0^\pi|g(t)|^2\frac{1+\sin t}{2\sin t}\operatorname{d}\! t<\infty\right\}.
    \]
    Moreover, for every $f\in A^2$,
    \[
    \|f\|_{A^2}^2=\int_{0}^{\pi}|\mathcal F(Jf)(t)|^2\frac{1+\sin t}{2\sin t}\operatorname{d}\!t.
    \]
    \end{thm}
    
    \begin{proof}
        Since $A^2\subseteq H^2$ we immediately deduce from Theorem \ref{t:H^2_PW} that for any $f$ in $A^2$ we have $\supp\mathcal F(Jf)\subseteq [0,\pi]$. Also, by \eqref{eq:f_representation_spectral}, we get
        \begin{align}\label{eq:Bergman_Plancherel}
        \begin{split}
              \|f\|^2_{A^2}&=\sum_{n=0}^{+\infty}\sum_{m=-\infty}^{+\infty}|f(m+in)|^2=\sum_{n=0}^{+\infty}\int_0^\pi |g(t)|^2 \Big|\frac{\cos t}{1+\sin t }\Big|^{2n}\operatorname{d}\!t\\
            &=\int_0^\pi |g(t)|^2 \sum_{n=0}^{+\infty} \Big|\frac{\cos t}{1+\sin t }\Big|^{2n}\operatorname{d}\!t=\int_0^\pi |g(t)|^2 \frac{1+\sin t}{2\sin t}\operatorname{d}\!t,
        \end{split}
        \end{align}
        where for the last identity we used the fact that $|\frac{\cos t}{1+\sin t}|<1$ for $0<t<\pi$. Thus, the inclusion $\mathcal F(J(A^2))\subseteq X$ follows. Conversely, given $g$ in $X$ let us define
        \[
        f(m+in)=
    \frac{1}{\sqrt{2\pi}}\int_0^\pi
    g(t)e^{imt}
    \left(\frac{\cos t}{1+\sin t}\right)^n
    \operatorname{d}\!t.
        \]
    Then, $f$ is well-defined,  $f\in \operatorname{H^2(\mathbb Z^2_+)}$  and $\mathcal F(Jf)=g$. Moreover, \eqref{eq:Bergman_Plancherel} guarantees also that $f$ belongs to $A^2$. Hence, $X\subseteq \mathcal F(J(A^2))$.
    \end{proof}
    
    \begin{cor}\label{c:Bergman_kernel}
    The Bergman space $A^2$ is a reproducing kernel Hilbert space. For $z=m+in$ and $w=u+iv$, its reproducing kernel
    is given by
    \[
    B_w(z)=
    C(z-\overline w)-C(z-\overline w+2i)
    =
    \frac{1}{2\pi}
    \int_0^\pi
    e^{i(m-u)t}
    \left(\frac{\cos t}{1+\sin t}\right)^{n+v}
    \frac{2\sin t}{1+\sin t}
    \operatorname{d}\!t,
    \]
    where $C$ is the Cauchy kernel.
    \end{cor}
    
    \begin{proof}
    It is clear by the previous theorem that $A^2$ is a Hilbert space. Moreover, 
    \[
    |f(z)|\leqslant \Big(\sum_{n=0}^{+\infty}\sum_{m=-\infty}^{+\infty}|f(m+in)|^2\Big)^{\frac 12}=\|f\|_{A^2},\quad  f\in A^2,\, z\in\mathbb Z^2_+.
    \]
    Thus, $A^2$ is a reproducing kernel Hilbert space. Let
    \begin{equation*}
         \beta_w(t)=\frac{1}{\sqrt{2\pi}} e^{-iut}\Big(\frac{\cos t}{1+\sin t}\Big)^v\Big(\frac{2\sin t}{1+\sin t}\Big)\chi_{[0,\pi]}(t).
    \end{equation*}
    Then, $\beta_w\in X,$ where $X$ is the space introduced in Theorem \ref{t:B^2_PW}. In particular, $\mathcal{F}^{-1}\beta_w\in A^2$, and for every $f\in A^2$ it satisfies
    \begin{align*}
        \langle f,\mathcal{F}^{-1}\beta_w\rangle_{A^2}=\int_0^\pi\mathcal{F}f(t)\overline{\beta_w(t)}\frac{1+\sin t}{2\sin t}\operatorname{d}\!t=\frac{1}{\sqrt{2\pi}}\int_0^\pi\mathcal{F}f(t)e^{iut}\Big(\frac{\cos t}{1+\sin t}\Big)^v\operatorname{d}\!t=f(w),
        \end{align*}
        where in the last equality we used the reproducing kernel Hilbert space structure of $H^2$, see Theorem \ref{t:H^2_PW}. Then, $B_w:=\mathcal{F}^{-1}\beta_w$ is the kernel of $A^2$ centered in~$w$. We explicitly compute it. 

   For $m+in\in\mathbb{Z}^2_+,$
    \[
    B_w(m+in)=\frac{1}{2\pi}\int_0^{\pi}  e^{i(m-u)t}\Big(\frac{\cos t}{1+\sin t}\Big)^{n+v}\Big(\frac{2\sin t}{1+\sin t}\Big) \operatorname{d}\!t.
    \]
    Since $\frac{2\sin t}{1+\sin t}=1-\big(\frac{\cos t}{1+\sin t}\big)^2$, we obtain 
    \begin{align*}
    B_w(m+in)&=\frac{1}{2\pi}\int_0^{\pi}  e^{i(m-u)t}\Big(\frac{\cos t}{1+\sin t}\Big)^{n+v}\operatorname{d}\!t-\frac{1}{2\pi}\int_0^{\pi}  e^{i(m-u)t}\Big(\frac{\cos t}{1+\sin t}\Big)^{n+v+2} \operatorname{d}\!t\\
    &=C(m-u+i(n+v))-C(m-u+i(n+v+2))\\
    &=C(z-\overline w)-C(z-\overline{w}+2i).
    \end{align*}
    \end{proof}
    We now briefly discuss weighted Bergman spaces. There are at least two natural
    ways to introduce weights in the present setting. The first one uses powers of the diagonal Bergman
    kernel as weights for the counting measure on \(\mathbb Z^2_+\). Thus, for a
    real parameter \(\nu\), we define \(A^2_\nu\) as the space of all discrete
    holomorphic functions on \(\mathbb Z^2_+\) such that
    \[
    \|f\|^2_{A^2_\nu}=\sum_{n=0}^{+\infty}\sum_{m=-\infty}^{+\infty}|f(m+in)|^2 B^\nu_{m+in}(m+in).
    \]
    A second construction is suggested by Theorem \ref{t:B^2_PW}. Given a real parameter $\mu$, we consider the weighted space 
    \[
    X_\mu=\left\{g: \supp g\subseteq [0,\pi], \int_0^\pi|g(t)|^2\left(\frac{1+\sin t}{2\sin t} \right)^\mu\operatorname{d}\! t<\infty\right\}.
    \]
    and then define the weighted Bergman space $\widetilde A^2_\mu=\mathcal F^{-1}(X_\mu)$. Similarly to the classical case, the two constructions are closely related.  More precisely, for $\mu>0$, the
    spectral scale $\widetilde A^2_\mu$ agrees, up to equivalence of norms, with
    the diagonal-kernel scale $A^2_{\frac{1-\mu}{2}}$.
    
    \begin{prop}
    The norms of \(\widetilde A^2_\mu\) and \(A^2_{\frac{1-\mu}{2}}\) are equivalent for $\mu>0$.
    \end{prop}
    \begin{proof}
    By the identities
    \[
    \frac{1+\sin t}{2\sin t}
    =
    \Big(
    1-\Big(\frac{\cos t}{1+\sin t}\Big)^2
    \Big)^{-1},
    \]
    and 
     \begin{equation}\label{eq:binomial}
    (1-s)^{-\mu}
    =
    \sum_{k=0}^{+\infty}
    \frac{\Gamma(k+\mu)}{\Gamma(\mu)k!}s^k,
    \qquad |s|<1,
    \end{equation}
    for a.e. $t\in[0,\pi]$, we obtain
    \begin{equation*}
    \Big(\frac{1+\sin t}{2\sin t}\Big)^\mu
    =
    \sum_{k=0}^{+\infty}
    \frac{\Gamma(k+\mu)}{\Gamma(\mu)k!}
    \Big(\frac{\cos t}{1+\sin t}\Big)^{2k}.
    \end{equation*}
    Let $f\in \widetilde A^2_\mu$
    . Then, $f$ is of the form
    \[
    f(m+in)
    =
    \frac{1}{\sqrt{2\pi}}
    \int_0^\pi
    g(t)e^{imt}
    \Big(\frac{\cos t}{1+\sin t}\Big)^n
    \operatorname{d}\!t,
    \qquad n\geqslant0.
    \]
    for some $g$ in $X_\mu$. Notice also that $X_\mu\subseteq X_0$, so that, by Plancherel, for every fixed $n\geqslant 0$, 
    \[
    \sum_{m=-\infty}^{+\infty}|f(m+in)|^2
    =
    \int_0^\pi
    |g(t)|^2
    \Big|\frac{\cos t}{1+\sin t}\Big|^{2n}
    \operatorname{d}\!t.
    \]
    Therefore,
    \begin{align*}
    \int_0^\pi
    |g(t)|^2
    \Big(\frac{1+\sin t}{2\sin t}\Big)^\mu
    \operatorname{d}\!t &=
    \sum_{n=0}^{+\infty}
    \frac{\Gamma(n+\mu)}{\Gamma(\mu)n!}
    \int_0^\pi
    |g(t)|^2
    \Big|\frac{\cos t}{1+\sin t}\Big|^{2n}
    \operatorname{d}\!t \\
    &=
    \sum_{n=0}^{+\infty}\sum_{m=-\infty}^{+\infty}
    \frac{\Gamma(n+\mu)}{\Gamma(\mu)n!}
    |f(m+in)|^2.
    \end{align*}
    Observe that 
    \[
    \frac{\Gamma(n+\mu)}{\Gamma(\mu)n!}
    =
    \frac{1}{\Gamma(\mu)}\,n^{\mu-1}
    \Big(1+O\Big(\frac{1}{n}\Big)\Big)
    \]
    as $n\to+\infty$.
    On the other hand, by Lemma~\ref{l:kernel_asymptotic} and
    Corollary~\ref{c:Bergman_kernel},
    \[
    B_{m+in}(m+in)
    =
    \frac{1}{2\pi n(n+1)}+O(n^{-3})
    \]
    as $n\to+\infty$.
    Thus, it follows that the norms of $\widetilde A^2_\mu$ and $A^2_{\frac{1-\mu}{2}}$ are comparable.
    \end{proof}
    Notice that for $\mu=0$ one recovers the Hardy space, namely $\widetilde A^2_0=H^2$. For $\mu<0$, $\mu\notin \mathbb Z$, the binomial expansion \eqref{eq:binomial} still holds true, but its coefficients are no longer all positive. Therefore the preceding argument does not identify $\widetilde A^2_\mu$ with a weighted Bergman space of the form $A^2_{\frac{1-\mu}{2}}$.
    This is analogous to the classical picture, where the Hardy space appears as a
    threshold between weighted Bergman spaces and Hardy--Sobolev, or
    Dirichlet-type, spaces. A systematic study of this range of parameters and of
    the structural properties of this scale is beyond the scope of the present
    paper and will be pursued elsewhere.
    \section{Final remarks}\label{s:final}
    The results of this paper suggest that a Hardy--Bergman theory for discrete holomorphic functions cannot be a verbatim translation of the classical one.
    Several classical objects, such as inner-outer factorization and Blaschke
    products, do not seem to have immediate analogues in the discrete setting. One
    basic obstruction is that the pointwise product of two discrete holomorphic
    functions is not, in general, discrete holomorphic. Thus, some of the central
    tools of the continuous theory cannot simply be copied to the lattice.

    This should not necessarily be regarded as a defect of the discrete theory. On
    the contrary, it suggests that the most natural questions may be those which are
    intrinsic to the discrete setting. Rather than trying to reproduce every
    classical construction, one may look for phenomena which are genuinely caused by
    the lattice geometry. The parity effects appearing in the present paper are a
    first example of this. The boundary sampling identity in
    Section~\ref{s:sampling}, for instance, shows that the $H^2$-norm splits
    exactly between the even and the odd boundary sublattices, a feature with no
    direct analogue in the continuous upper half-plane.
    
    At the same time, the approximation result of Section~\ref{s:approximation}
    shows that the discrete Hardy spaces considered here are compatible with the
    classical Hardy space in a suitable scaling limit. Thus the discrete theory is
    not merely an isolated lattice model: it retains a meaningful connection with
    the continuous one, while also exhibiting new features which disappear in the
    limit.
    
    Zero sets, uniqueness, sampling, and interpolation appear to be particularly
    natural problems in this direction. They are sensitive to the geometry of the
    lattice and may reveal phenomena which have no continuous counterpart. Another
    natural direction is to study growth properties and quantitative uniqueness
    principles for discrete holomorphic functions. Since their real and imaginary
    parts are discrete harmonic on the two diagonal sublattices, these questions are
    closely related to the theory of harmonic functions on graphs and lattices. In
    this direction, a striking
    strengthening of the Liouville theorem for discrete harmonic functions on
    $\mathbb Z^2$ was proved in \cite{Malinnikova} and later generalized to periodic planar graphs in \cite{Bou}. It would be
    interesting to understand whether the additional discrete holomorphic structure,
    beyond discrete harmonicity, leads to stronger rigidity, new quantitative
    uniqueness principles, or phenomena depending on the parity structure of the
    square lattice.
    
    It would also be interesting to understand the endpoint duality of these spaces.
    In the classical theory, the endpoint duality of $H^1$ is described in terms of
    $BMOA$, or equivalently in terms of boundary values in $BMO$. It would be interesting to further investigate this duality in the discrete setting. A possible natural approach is to use the boundary identification of $H^1$, together
    with the discrete Szeg\H{o} projection and the Cauchy kernel, to look for a
    lattice-adapted replacement of $BMO$.

    \appendix
    \section{A proof of Lemma \ref{l:kernel_asymptotic}}
    In this appendix we prove the asymptotic expansion \eqref{eq:asymptotic} of the Cauchy kernel, from which the estimates in \eqref{eq:K_norm_p} easily follow. In particular, we prove that
    \begin{equation*}
    C(m+in)=
    \frac{i}{2\pi}\left(
    \frac{1}{m+in}
    -
    (-1)^{m+n}\frac{1}{m-in}
    \right)
    +O\!\left((m^2+n^2)^{-\frac32}\right),
    \end{equation*} for $m+in$ in $\mathbb{Z}^2\setminus\{0\}.$  We deal with the different cases separately.
    \subsection{The boundary case $n=0$}
    In this situation the kernel $C$ is exactly computed, indeed
    \[ C(m)=\frac 1{2\pi}\int_0^\pi e^{imt}\operatorname{d}\!t=\begin{dcases*}    \frac 12,&\text{if }$m=0$,\\
    0,&\text{if }$m\in 2\mathbb Z\setminus\{0\}$,\\
    \frac i{\pi m},&\text{if }$m\in 2\mathbb Z+1$,
    \end{dcases*}
    \]
    and the conclusion follows.
    \subsection{The case $m\neq0,n\neq 0$}
    Let us assume that $m\neq0,n\neq 0$. It is sufficient to prove the statement for $n\geqslant  1$, indeed if $n\leqslant -1$, then, by~\eqref{eq:kernel_negative}, we have
    \begin{align*}
        C(m+in)&=(-1)^{m+n+1}C(m-in)\\
        &=(-1)^{m+n+1}\frac i{2\pi}\Big( \frac 1{m-in}-(-1)^{m+n}\frac 1{m+in}\Big)+O((m^2+n^2)^{-\frac{3}{2}})\\
        &=\frac i{2\pi}\Big(\frac 1{m+in}-(-1)^{m+n}\frac 1{m-in}\Big)+O((m^2+n^2)^{-\frac{3}{2}}).
    \end{align*}
    
    Let $\delta>0$ be fixed. Let $\chi$ in $ C^\infty((0,\pi))$ be a cutoff function such that $\chi(u)=\chi(\pi-u)$, $\chi|_{[0,\delta]}\equiv 1$ and $\chi|_{[2\delta,\pi-2\delta]}\equiv 0$. We set the following notation:
    \[ 
        I_\delta(m,n)=\int_{0}^{2\delta}\chi(t) e^{imt} \Big(\frac{\cos t} {1+\sin t} \Big)^n\operatorname{d}\!t,\qquad
        T_\delta(m,n)=\int_{\delta}^{\pi-\delta} (1-\chi(t))e^{imt} \Big(\frac{\cos t} {1+\sin t} \Big)^n\operatorname{d}\!t.
    \]
    By applying a simple change of variable, we can split the integral of $C$ as follows
        \begin{align*}
            2\pi C(m+in)&=\int_{0}^{2\delta} \chi(t)e^{imt} \Big(\frac{\cos t} {1+\sin t} \Big)^n\operatorname{d}\!t+\int_{\pi-2\delta}^{\pi} \chi(t) e^{imt} \Big(\frac{\cos t} {1+\sin t} \Big)^n\operatorname{d}\!t+T_\delta(m,n)\\
        &=I_\delta(m,n)
        +(-1)^{m+n} \overline{I_\delta(m,n)} +T_\delta(m,n)
        \end{align*}
        
        We first estimate $T_\delta$, which will be a negligible term. 
        \begin{prop}
          For every $K\geqslant1$  We have $T_\delta(m,n)= O ((m^2+n^2)^{-K})$.
        \end{prop}
        \begin{proof}
            We divide the proof in two cases. Assume first $n\geqslant c|m|$, for some $c>0$. In this region, we have that $n\approx (m^2+n^2)^{\frac 12}$, indeed $n\leqslant (m^2+n^2)^{\frac 12}\leqslant n\Big(\frac 1{c^2}+1\Big)^{\frac 12}$. 
     Then, we clearly have that
    \begin{align}\label{eq:tail}
      |T_\delta(m+in)|&\leqslant\int_{\delta}^{\pi-\delta} e^{-n\log\big|\frac{1+\sin t}{\cos t}\big|}\operatorname{d}\!t\leqslant e^{-nC_\delta}\leqslant e^{-\alpha(m^2+n^2)^{\frac12}C_\delta}.  
    \end{align}
    Thus, $|T_\delta(m+in)|= O (e^{-\alpha\sqrt{m^2+n^2}})$. In particular, we also have that  $T_\delta(m,n)= O ((m^2+n^2)^{-\frac K2})$ for every $K\geqslant 1$.
    
    Consider now the case $n\leqslant  c|m|$, for some $c>0$. Clearly, here we have that $m\approx (m^2+n^2)^{\frac 12}$. We can integrate by parts $K$ times the integral in $T_\delta$ and the boundary contribution vanishes every time because of $\eta=    1-\chi$, obtaining
    \[T_\delta(m,n)=\frac{(-1)^K}{(im)^K}\int_0^{\pi}e^{imt}(\eta\varphi_n)^{(K)}(t)\operatorname{d}\!t, 
    \]
    where
    \[
    \varphi_n(t)=\Big(\frac{\cos t} {1+\sin t} \Big)^n.
    \]
    The function
    $\varphi_n(z)$
    is holomorphic in the strip $\{z=x+iy\in\mathbb C: -\pi/2<x<3\pi/2 \}$. Let $\varepsilon>0$ small enough such that  $C_\varepsilon(t)$ is a circumference that strictly contains the interval $[\delta-\varepsilon,\pi/2+\varepsilon]$ and it is strictly contained in the strip $\{z=x+iy\in\mathbb C: -\pi/2<x<3\pi/2 \}$ and, moreover, that $\sup_{\zeta\in C_\varepsilon}|\varphi(\zeta)|\leqslant 1$
    (possible since $\sup_{[\delta,\pi/2]}\varphi=\varphi(\delta)<1$ and $\varphi$ is continuous),
    so that $\sup_{\zeta\in C_\varepsilon}|\varphi_n(\zeta)|=\big(\sup_{\zeta\in C_\varepsilon}|\varphi(\zeta)|\big)^n\leqslant 1$ uniformly in $n$. Here, $\varphi=\varphi_1$.
    By Cauchy's formula,
    \[
    \frac{d^j}{dt^j}(\varphi_n(t))=\frac{j!}{2\pi}\int_{C_\varepsilon}\frac{\varphi_n(\zeta)}{(\zeta-t)^{j+1}}\operatorname{d}\!\zeta.
    \]
    Hence,
    \[
    \Big|\frac{d^j}{dt^j}(\varphi_n(t))\Big|\leqslant j! \tilde c_{\varepsilon,j}\sup_{\zeta\in  C_\varepsilon}|\varphi_n(\zeta)|\leqslant j! c_{\varepsilon, j},\quad t\in[\delta,\pi/2].
    \]
    Using Leibniz' formula,
    \[(\eta \varphi_n)^{(K)}=\sum_{j=0}^{K}\binom{K}{j}\eta^{(j)}(\varphi_n)^{(K-j)},\]
    and therefore
    $|(\eta \varphi_n)^{(K)}(t)|\leqslant \tilde C_{\delta,\varepsilon,K}$.
    Consequently,
    \[|T_\delta(m,n)|\leqslant C_{\delta,\varepsilon,K}\frac 1{|m|^K}= O ((m^2+n^2)^{-\frac K2}),\]
    by the fact that $|m|\approx (m^2+n^2)^{\frac 12}$.
        \end{proof}
    In order to estimate the contribution of $I_\delta$, we need to recall a technical result.
    \begin{lem}\label{lem: Gamma} Let $\delta>0$. The following hold true.
    \begin{enumerate}
        \item Let \(F\in C_c^\infty([0,\infty))\). Assume that $F(u)=O(u^M)$ for $u\to 0$ and some $M>0$. Then 
        \[
    \int_0^{+\infty} e^{-zu}F(u)\operatorname{d}\!u
    =
    O(|z|^{-M-1}),
    \]
    uniformly for $\operatorname{Re} z\geqslant 
    1$.
    \item 
        Let $\theta\in C^\infty$ such that $\theta|_{[0,\delta]}\equiv 1$ and $\theta|_{[c\delta,+\infty)}\equiv 0$, for some $c>1$. For every $k\in\mathbb N$ and $K\geqslant1$, we have
        \[\int_0^{\infty} e^{-zu}u^k\theta(u)\operatorname{d}\!u=\frac {k!}{z^{k+1}}+ O (|z|^{-K}),\]
        uniformly for $\operatorname{Re} z\geqslant1$.
    \end{enumerate}
    \end{lem}
    \begin{proof}
        We start with the first item. 
        Integrating by parts \(M\) times,
    all boundary terms up to order \(M-1\) vanish at \(u=0\), while the
    contribution at \(+\infty\) is zero because \(F\) is compactly
    supported. Integrating by parts one time more, the resulting integral is bounded by
    \[\int_0^\infty e^{-zu}F(u)\operatorname{d}\!u=\frac{F^{(M)}(0)}{z^{M+1}}+\frac1{z^{M+1}}\int_0^\infty e^{-zu}F^{(M+1)}(u)\operatorname{d}\!u=O(|z|^{-M-1}),\]
    which proves the claim.
        
        We prove now the second part of the Lemma. We split the integration as
        \[\int_0^{+\infty} e^{-zu}u^k\theta(u)\operatorname{d}\!u= \int_0^{+\infty} e^{-zu}u^k\operatorname{d}\!u+\int_0^{+\infty} e^{-zu}u^k(1-\theta(u))\operatorname{d}\!u.\]
        The first integral can be reduced to a Gamma function by using a change of variable. Indeed,
        \[ 
        \int_0^{+\infty} e^{-zu}u^k\operatorname{d}\!u=\frac 1{z^{k+1}}\int_0^{+\infty} e^{-t}t^k\operatorname{d}\!t=\frac {k!}{z^{k+1}}.
        \]
        For the second term, we have that the function $g(u)=u^k(1-\theta(u))$ and all its derivatives vanish at $0$. We can then integrate by parts $K$ times getting
        \[
        \int_0^{+\infty} e^{-zu}u^k(1-\theta(u))\operatorname{d}\!u=\frac 1{z^K}\int_0^{+\infty} e^{-zu}g^{(K)}(u)\operatorname{d}\!u=O(|z|^{-K}).\]
    \end{proof}
    We are ready to estimate the integral $I_\delta$.
    \begin{prop}
        For $\delta>0$, we have that
    \[ I_\delta(m,n)=\frac 1{n-im}+ O \left(\frac 1{|m^2+n^2|^{\frac 32}}\right).\]
    \end{prop}
    \begin{proof}By applying the change of variable \[u=\log\left(\frac{1+\sin t}{\cos t}\right),\]
    inverted by $t(u)$, we get, for a suitable $\delta'=t(2\delta)>0$ and $\theta(u)=\chi(t(u))$,
    \begin{align*}
    I_\delta(m,n)&=\int_{0}^{2\delta}\chi(t) e^{imt} e^{-n\log\big(\frac{1+\sin t}{\cos t} \big)}\operatorname{d}\!t\\
        &=\int_{0}^{\delta'} \theta(u)e^{imt(u)} e^{-nu}\frac{\operatorname{d}\!u}{\cosh u}=\int_{0}^{\delta'} \theta(u)\frac{e^{im(t(u)-u)}}{\cosh u} e^{(im-n)u}\operatorname{d}\!u.
    \end{align*}
    
    By direct computation, one can observe that 
    \begin{equation}
        \label{eq: Taylor phi} \psi(t)=\log\left(\frac {1+\sin t}{\cos t}\right)=t+\frac 16 t^3+ O (t^5).
    \end{equation}
    Furthermore, since $\psi(t)$ is odd, its inverse $t(u)$ as to be odd. Hence, the Taylor polynomial of $t(u)$ as to be of the form $t(u)=c_1 u+c_3 u^3 + O (u^5)$.
    Substituting this expansion  in~\eqref{eq: Taylor phi}, we get
    \begin{align*}
        u=\psi(t(u))&=c_1 u+\left( \frac 16 c_1^3+c_3\right) u^3+ O (u^5)\qquad\Longrightarrow\qquad t(u)-u=-\frac 16 u^3 + O (u^5).
    \end{align*}
    Now, by using the Taylor expansions of the exponential and of $1/\cosh$,
    we get
    \begin{align*}
        \frac{e^{im(t(u)-u)}}{\cosh u}
        &=\left(1-\frac 16 im u^3 + O (|m|u^5)\right)\left(1-\frac 12 u^2+ O (u^4)\right)\\
        &=1-\frac 12 u^2-\frac 16 im u^3+ O (u^4)+ O (|m|u^5).
    \end{align*}  
    By observing that $\theta$ is compactly supported and using Lemma~\ref{lem: Gamma}, we conclude that
        \begin{align*}
            I_\delta(m,n)=&\int_{0}^{\infty} \left(1-\frac 12 u^2-\frac 16 im u^3+ O (u^4)+ O (|m|u^5)\right) \theta(u)e^{(im-n)u}\operatorname{d}\!u\\
            =&\frac 1{n-im}-\frac 1{(n-im)^3}-im \frac 1{(n-im)^4}+ O \left(\frac 1{|m^2+n^2|^{\frac 52}}\right)+ O \left(\frac {|m|}{|m^2+n^2|^3}\right)\\
            =&\frac 1{n-im}-\frac {n}{(n-im)^4}+ O \left(\frac 1{|m^2+n^2|^{\frac 52}}\right)=\frac 1{n-im}+ O \left(\frac 1{|m^2+n^2|^{\frac 32}}\right).
        \end{align*}
        \end{proof}
    \subsection{The case $m=0$}
    We start by recalling that $\varphi(\pi-t)=-\varphi(t)$, so that, when $m=0$, we have
    \[ C(in)=\frac 1{2\pi}\int_0^\pi \varphi(t)^n\operatorname{d}\!t=\begin{dcases*}
    0,&\text{if }$n\in 2\mathbb Z+1$,\\
    \frac 1\pi\int_0^{\frac \pi 2}\Big(\frac{\cos t} {1+\sin t} \Big)^n\operatorname{d}\!t,&\text{if }$n\in 2\mathbb Z$.
    \end{dcases*}\]
    The integral that we have to estimate can be done exactly as in the previous case, using that $(\cosh u)^{-1}=1+O(u^2)$
    and getting
    \begin{align*}
        \frac 1\pi\int_0^{\frac \pi 2}\Big(\frac{\cos t} {1+\sin t} \Big)^n\operatorname{d}\!t=\frac 1\pi\int_0^{+\infty}\frac {e^{-nu}}{\cosh u}\operatorname{d}\!u=\frac 1\pi\int_0^{+\infty}e^{-nu}(1+O(u^2))\operatorname{d}\!u=\frac 1{\pi n}+O\Big(\frac 1{n^3}\Big),
    \end{align*}
    as we wanted to prove.
    \bibliographystyle{amsalpha}
    \bibliography{DHFbib}

\providecommand{\bysame}{\leavevmode\hbox to3em{\hrulefill}\thinspace}
\providecommand{\MR}{\relax\ifhmode\unskip\space\fi MR }
\providecommand{\MRhref}[2]{%
  \href{http://www.ams.org/mathscinet-getitem?mr=#1}{#2}
}
\providecommand{\href}[2]{#2}
\begin{thebibliography}{BRCG25}

\bibitem[AV22]{AV}
D.~Alpay and D.~Volok, \emph{Discrete analytic {S}chur functions}, Proc. Amer.
  Math. Soc. \textbf{150} (2022), no.~5, 2145--2152.

\bibitem[Bea73]{beals}
R.~Beals, \emph{Advanced mathematical analysis. {P}eriodic functions and
  distributions, complex analysis, {L}aplace transform and applications},
  Graduate Texts in Mathematics, vol. No. 12, Springer-Verlag, New
  York-Heidelberg, 1973.

\bibitem[BG16]{BG16}
A.~I. Bobenko and F.~G\"unther, \emph{Discrete complex analysis on planar
  quad-graphs}, Advances in discrete differential geometry, Springer, [Berlin],
  2016, pp.~57--132.

\bibitem[BLMS22]{Malinnikova}
L.~Buhovsky, A.~Logunov, E.~Malinnikova, and M.~Sodin, \emph{A discrete
  harmonic function bounded on a large portion of {$\Bbb Z^2$} is constant},
  Duke Math. J. \textbf{171} (2022), no.~6, 1349--1378.

\bibitem[BnK19]{BK}
R.~Ba\~nuelos and M.~Kwa\'snicki, \emph{On the {$\ell^p$}-norm of the discrete
  {H}ilbert transform}, Duke Math. J. \textbf{168} (2019), no.~3, 471--504.

\bibitem[Bol68]{Bol68}
J.~C. Bolen, \emph{A reproducing kernel function and convergence properties for
  discrete analytic functions}, Ph.D. thesis, Texas Christian University, 1968.

\bibitem[BRCG25]{Bou}
A.~Bou-Rabee, W.~Cooperman, and S.~Ganguly, \emph{Unique continuation on planar
  graphs}, Discrete Anal. \textbf{2025} (2025), 12.

\bibitem[Che16]{Che16}
D.~Chelkak, \emph{Robust discrete complex analysis: a toolbox}, Ann. Probab.
  \textbf{44} (2016), no.~1, 628--683.

\bibitem[CLR23]{CLR23}
D.~Chelkak, B.~Laslier, and M.~Russkikh, \emph{Dimer model and holomorphic
  functions on t-embeddings of planar graphs}, Proc. Lond. Math. Soc. (3)
  \textbf{126} (2023), no.~5, 1656--1739.

\bibitem[CS11]{CS11}
D.~Chelkak and S.~Smirnov, \emph{Discrete complex analysis on isoradial
  graphs}, Adv. Math. \textbf{228} (2011), no.~3, 1590--1630.

\bibitem[CS12]{CS12}
\bysame, \emph{Universality in the 2{D} {I}sing model and conformal invariance
  of fermionic observables}, Invent. Math. \textbf{189} (2012), no.~3,
  515--580.

\bibitem[DM71]{DM71}
C.~R. Deeter and C.~W. Mastin, \emph{The discrete analog of a minimum problem
  in conformal mapping}, Indiana Univ. Math. J. \textbf{20} (1970/71),
  355--367.

\bibitem[DP68]{DP68}
R.~J. Duffin and Elmor~L. Peterson, \emph{The discrete analogue of a class of
  entire functions}, J. Math. Anal. Appl. \textbf{21} (1968), 619--642.

\bibitem[Duf56]{D}
R.~J. Duffin, \emph{Basic properties of discrete analytic functions}, Duke
  Math. J. \textbf{23} (1956), 335--363.

\bibitem[Duf68]{Duf68}
\bysame, \emph{Potential theory on a rhombic lattice}, J. Combinatorial Theory
  \textbf{5} (1968), 258--272.

\bibitem[Fer44]{F}
J.~Ferrand, \emph{Fonctions pr\'eharmoniques et fonctions pr\'eholomorphes},
  Bull. Sci. Math. (2) \textbf{68} (1944), 152--180.

\bibitem[H\"03]{hormander1983analysis}
Lars H\"ormander, \emph{The analysis of linear partial differential operators.
  {I}}, Classics in Mathematics, Springer-Verlag, Berlin, 2003, Distribution
  theory and Fourier analysis, Reprint of the second (1990) edition.

\bibitem[HMW73]{HMW}
R.~Hunt, B.~Muckenhoupt, and R.~Wheeden, \emph{Weighted norm inequalities for
  the conjugate function and {H}ilbert transform}, Trans. Amer. Math. Soc.
  \textbf{176} (1973), 227--251.

\bibitem[Isa41]{Isa41}
R.~P. Isaacs, \emph{A finite difference function theory}, Univ. Nac.
  Tucum{\'a}n. Revista A \textbf{2} (1941), 177--201.

\bibitem[Isa52]{Isa52}
\bysame, \emph{Monodiffric functions}, Construction and applications of
  conformal maps. {P}roceedings of a symposium, National Bureau of Standards
  Applied Mathematics Series, vol. No. 18, U.S. Govt. Printing Office,
  Washington, DC, 1952, pp.~257--266.

\bibitem[Ken02]{Ken02}
R.~Kenyon, \emph{The {L}aplacian and {D}irac operators on critical planar
  graphs}, Invent. Math. \textbf{150} (2002), no.~2, 409--439.

\bibitem[Lov04]{Lov04}
L.~Lov{\'a}sz, \emph{Discrete analytic functions: an exposition}, Surveys in
  Differential Geometry \textbf{9} (2004), no.~1, 241--273.

\bibitem[Mer01]{Mer01}
C.~Mercat, \emph{Discrete {R}iemann surfaces and the {I}sing model}, Comm.
  Math. Phys. \textbf{218} (2001), no.~1, 177--216.

\bibitem[Mer08]{Mer08}
\bysame, \emph{Discrete complex structure on surfel surfaces}, Discrete
  geometry for computer imagery, Lecture Notes in Comput. Sci., vol. 4992,
  Springer, Berlin, 2008, pp.~153--164.

\bibitem[MM26]{MM}
A.~Monguzzi and M.~Monti, \emph{On discrete holomorphic {P}aley--{W}iener
  spaces and sampling on the square lattice}, Trans. Amer. Math. Soc.
  \textbf{379} (2026), no.~7, 5237--5259.

\bibitem[Mug80]{Mug80}
D.~H. Mugler, \emph{The discrete {P}aley-{W}iener theorem}, J. Math. Anal.
  Appl. \textbf{75} (1980), no.~1, 172--179.

\bibitem[PW87]{PW}
R.~E. A.~C. Paley and N.~Wiener, \emph{Fourier transforms in the complex
  domain}, American Mathematical Society Colloquium Publications, vol.~19,
  American Mathematical Society, Providence, RI, 1987, Reprint of the 1934
  original.

\bibitem[RT10]{distributionsRT}
M.~Ruzhansky and V.~Turunen, \emph{Pseudo-differential operators and
  symmetries, background analysis and advanced topics}, Pseudo-Differential
  Operators, Birkhäuser Basel, 2010.

\bibitem[Sko13]{Sko13}
M.~Skopenkov, \emph{The boundary value problem for discrete analytic
  functions}, Adv. Math. \textbf{240} (2013), 61--87.

\bibitem[Smi10]{Smi10}
S.~Smirnov, \emph{Discrete complex analysis and probability}, Proceedings of
  the International Congress of Mathematicians. Volume I (New Delhi), Hindustan
  Book Agency, 2010, pp.~595--621.

\bibitem[ZD77]{ZD}
D.~Zeilberger and H.~Dym, \emph{Further properties of discrete analytic
  functions}, J. Math. Anal. Appl. \textbf{58} (1977), no.~2, 405--418.

\bibitem[Zei77a]{Zei77a}
D.~Zeilberger, \emph{A discrete analogue of the {P}aley-{W}iener theorem for
  bounded analytic functions in a half plane}, J. Austral. Math. Soc. Ser. A
  \textbf{23} (1977), no.~3, 376--378.

\bibitem[Zei77b]{Zei77b}
\bysame, \emph{Discrete analytic functions of exponential growth}, Trans. Amer.
  Math. Soc. \textbf{226} (1977), 181--189.

\bibitem[Zei77c]{Zei77c}
\bysame, \emph{A new approach to the theory of discrete analytic functions}, J.
  Math. Anal. Appl. \textbf{57} (1977), no.~2, 350--367.

\end{thebibliography}
    \end{document}